\documentclass{amsart}

\usepackage{graphicx} 
\usepackage{amsmath}
\usepackage{amssymb}
\usepackage{amsthm}
\usepackage{amsfonts}
\usepackage{appendix}
\usepackage{tikz-cd}
\usepackage{tikz}        
\usepackage{bm}

\graphicspath{{./figures/}}

\everymath{\displaystyle}

\newtheorem{prop}{Proposition}

\newtheorem{thm}{Theorem}
\newtheorem*{Mthm}{Main Theorem}
\newtheorem{lem}{Lemma}
\newtheorem{cor}{Corollary}
\theoremstyle{definition}
\newtheorem{eg}{Example}
\newtheorem{dfn}{Definition}
\newtheorem{rmk}{Remark}

\newcommand{\C}{\mathbb{C}}
\newcommand{\R}{\mathbb{R}}
\newcommand{\Z}{\mathbb{Z}}
\newcommand{\h}{\mathfrak{h}}

\title{Bubbling limits of non collapsing polarized K3 surfaces}
\author{Itsuki Tazoe }

\thanks{Department of Mathematics, Kyoto University, Kyoto 606-8285. JAPAN}
\email{tazoe.itsuki.52n@st.kyoto-u.ac.jp}

\begin{document}

\maketitle

\begin{abstract}
    We give an explicit and complete description of bubbling limits of a non-collapsing limit of polarized K3 surfaces in terms of the period mapping.
    In particular, we show that bubbling limits only depend on algebro-geometric data of the given family.
    As a corollary, this gives an affirmative answer to a conjecture of de Borbon--Spotti and confirms that Odaka's algebro-geometric candidate gives genuine bubbling limits in K3 surfaces case.
\end{abstract}

\tableofcontents

\section{Introduction}
Let $\{(X_j,g_j)\}_{j=1}^\infty$ be a sequence of compact K\"ahler manifolds which converges to a compact metric space $(X_\infty,d_\infty)$ in the Gromov-Hausdorff topology.
Assume that each $(X_j,g_j)$ is polarized, i.e. there is a positive line bundle $(L_j, h_j)$ on $X_j$ such that the K\"ahler form $\omega_j$ of $g_j$ is equal to the first Chern form $c_1(L_j, h_j)$.
If diameters, volumes and Ricci curvatures of the sequence are uniformly bounded, the limit space $(X_\infty, d_\infty)$ is called a \textit{non-collapsing limit} and it is known that the limit space is naturally a normal projective variety and the metric $d_\infty$ is defined by a weak K\"ahler metric $g_\infty$ by the seminal work of Donaldson-Sun \cite{DS}\cite{DS2}.
Furthermore, for a diverging sequence of real numbers $\{c_j\}_{j=1}^\infty$, $\{(X_j, c_j^2 g_j, x_j)\}_{j=1}^\infty$ contains a convergent subsequence, if base points $x_j \in X_j$ are chosen suitably, and the limit is an affine variety.
Such a \textit{rescaled limit} is called a \textit{bubbling limit} of the sequence $\{(X_j,g_j)\}_{j=1}^\infty$.
Non-collapsing limits and its bubbling limits are closely related to algebraic geometry but the relation is still mysterious.

For the case of polarized K3 surfaces, non-collapsing limits are well-understood.
In fact, the following is shown in Kobayashi-Todorov \cite{KT} (see also Proposition 6.7. in \cite{OO}).
\begin{thm}[Theorem 8. \cite{KT}, see also Proposition 6.7 in \cite{OO} and \cite{Anderson}]
    Let 
\begin{equation}
    (\mathcal{X}, \mathcal{L}) \to \Delta
\end{equation}
be a flat proper family of polarized K3 surfaces over the unit disc $\Delta \subset \C$.
Assume that the central fiber $X_0$ has ADE-type singularities.
For each fiber $(X_t, L_t)$, take the Ricci flat K\"ahler metric $g_t$ with the K\"ahler form $\omega_t \in c_1(L_t)$ (for the central fiber $X_0$, $g_0$ is taken as an orbifold metric).
Then
\begin{equation}
    (X_t, g_t) \to (X_0,g_0)
\end{equation} 
in the sense of Gromov-Hausdorff.
\end{thm}

By the above theorem, non-collapsing limits of K3 surfaces are completely identified by algebro-geometric data of the family $(\mathcal{X}, \mathcal{L}) \to \Delta$.
For bubbling limits, as a corollary of Anderson \cite{Anderson}, Nakajima \cite{Nakajima}, Bando-Kasue-Nakajima \cite{BKN} and Bando \cite{Bando}, it is known that they are $\C^2$ or $\C^2/ \Gamma$ with the standard metric where $\Gamma \subset \mathrm{SL}(2;\C)$ if they are flat or complete hyperk\"ahler $4$-manifolds (possibly having orbifold singularities) the so-called ALE hyperk\"ahler gravitational instantons if they are non-flat.
ALE hyperk\"ahler gravitational instantons are completely classified by Kronheimer \cite{Kro1}\cite{Kro2}.
However, it has remained open to describe or \textit{determine explicitly} which ALE gravitational instantons appear as bubbling limits for a given family of polarized K3 surfaces.
The following classification theorem of bubbling limits of multi-Eguchi-Hanson spaces, which is a typical example of ALE gravitational instantons and provides a \textit{local model} of non-collapsing limits of K3 surfaces, is proven in de Borbon-Spotti \cite{dBS} :
\begin{thm}[Theorem 3. \cite{dBS}]
    Let $z_j : \Delta \to \C, \ (j=0,\ldots n)$ be holomorphic functions on the unit disc $\Delta = \{ z \in \C \mid |z| <1\}$ such that $z_j(0) =0, (j=0,\ldots , n)$ and $z_j(t) \neq z_k(t)$ if $j \neq k$ for $t\neq 0$.
    Consider the family $\mathcal{X} \to \Delta$ of affine surfaces defined by
    \begin{equation}
        \mathcal{X} = [\Pi_{j=0}^n (z-z_j(t)) = xy] \subset \C^3_{(x,y,z)} \times \Delta_t.
    \end{equation}
    Equip the multi-Eguchi-Hanson metric $g_t$ on $X_t = [\Pi_{j} (z-z_j(t)) = xy] \subset \C^3_{(x,y,z)}$. 
    Then its non-cone bubbling limits one to one correspond to vertices of a tree $\mathcal{T}$ constructed from $\{z_j(t)\}$.
\end{thm}
See section \ref{sec:Examples}, or section 3.1. of the original paper\cite{dBS}, for more details (in particular, the construction of the tree $\mathcal{T}$).
A key point of the theorem is that the tree $\mathcal{T}$ is completely and \textit{explicitly determined} by $\{z_j(t)\}$, the algebro-geometric data describing the family $\mathcal{X} \to \Delta$.
Also note that, in Odaka \cite{odaka25}, a \textit{candidate} of a purely algebro-geometric construction of bubbling limits is proposed.

In the present paper, a complete and explicit description of bubbling limits of non-collapsing limits of polarized K3 surfaces is given in terms of the period map of polarized K3 surfaces, the algebro-geometric data describing the given family. 
This gives an affirmative answer to Conjecture 1. in \cite{dBS} for the case of K3 surfaces including $D_n$ and $E_n$ type singularities and confirms Odaka's purely algebro-geometric candidate gives genuine bubbling limits in this case.

To state our main theorem, some notions, \textit{period bubbling tree} and \textit{metric bubbling tree}, are needed (see Section \ref{section:bubbling trees} for details).
Consider a flat proper family 
\begin{equation}
    (\mathcal{X}, \mathcal{L}) \to \Delta
\end{equation}
of polarized K3 surfaces over the unit disc $\Delta \subset \C$ with smooth general fibers and a singular central fiber.
Take a minimal simultaneous resolution 
\begin{equation}
    \widetilde{\mathcal{X}} \to \mathcal{X}
\end{equation}
of the family after a suitable base change (if it is necessary).
Let $\widetilde{\mathcal{L}}$ be the pullback of $\mathcal{L}$ on $\widetilde{\mathcal{X}}$.
By taking Ricci-flat K\"ahler metrics $\omega_t \in c_1(L_t)$ on $X_t$, holomorphic volume forms $\Omega_t$ on $X_t$ and a simultaneous marking 
\begin{equation}
    \alpha : H^2(X_t ;\Z) \to L,
\end{equation}
where $L$ is the K3 lattice, the so-called \textit{period mapping} is defined as follows:
\begin{equation}
\begin{aligned}
    \mathcal{P} :\Delta &\to L_{2d,\C}\  (:= (\alpha(c_1(L_t)))^{\perp} \subset L)\\
    t &\mapsto \alpha([\Omega_t]).
\end{aligned}
\end{equation}
Let $x_0 \in X_0 \cong 0 \in \C^2/\Gamma$ be a singularity.
Let $E_1,\ldots, E_n$ be the irreducible components of the exceptional divisor of the minimal resolution.
Then the complexification of a sub-lattice 
\begin{equation}
\h := \mathrm{Span}_{\Z} \{\theta_1,\ldots,\theta_n\} \subset L_{2d}
\end{equation}
is isometric to a Cartan sub-algebra of a simple complex Lie algebra of the same ADE-type with $x_0 \in X_0 \cong 0 \in \C^2/\Gamma$, where $\theta_j = \alpha([E_j])$.
Let
\begin{equation}
    \zeta_{\h} = \pi_\h \circ \mathcal{P}
\end{equation}
where $\pi_\h$ is the orthogonal projection to $\h_\C \subset L_{2d, \C}$.
A \textit{period-bubbling tree} $\mathcal{PBT}_\zeta$ of $\zeta$ is defined as follows:
\begin{itemize}
    \item the root is $[\zeta_r] \in \mathbb{P}(\h_\C)$ if $\zeta (t) = t^k \zeta_r + O(t^{k+1})$
    \item a vertex $[\zeta_w]\in \mathbb{P}(\h_{\zeta_w})$ is a child of $[\zeta_v] \in \mathbb{P}(\h_{\zeta_v})$ if and only if $\pi_{\zeta_{w}} \circ \zeta(t) = t^l \zeta_w + O(t^{l+1})$ where $\pi_{\zeta_{w}}$ is the projection to a subspace $\h_{\zeta_{w}}$ spanned by a maximal irreducible sub-root system of $\{\theta \in \h_{\zeta_v} \mid \langle \theta,\zeta_{v}\rangle = 0\}$ containing $\zeta_w$.
\end{itemize}
Here, a tree is a finite poset $T$ such that $u,v \geq w$ implies either $u \geq v$ or $v \geq u$ and having a unique maximal element.
The maximal element is called as the \textit{root} of $T$.
A vertex $v$ is the parent of $w$ if $v = \min \{u \in T \mid u > w\}$ and $w$ is a child of $v$ if $v$ is the parent of $w$.

Another tree $\mathcal{MBT}_{x_0}$, a \textit{metric-bubbling tree}, which is a tree of \textit{genuine bubbling limits at $x_0$}, is constructed as follows (see also section \ref{sec:mbt}).
As a set, $\mathcal{MBT}_{x_0}$ consists of equivalence classes of pairs of sections $\sigma : \Delta \to \mathcal{X}$ through $x_0$ and scaling factors $c(t) : \Delta \to \R_{> 0}$ which gives an affine ALE instanton $B$ with $(X_t, c_t^2g_t, \sigma_t) \to B$. Precisely, 
\begin{equation}
    \mathcal{MBT}_{x_0} = \{(\sigma,c) \mid B =\lim (X_t, c^2(t)g_t, \sigma(t))\ \text{is a non-cone bubbling limits} \}/\sim
\end{equation}
where $(\sigma_1, c_1) \sim (\sigma_2,c_2)$ if and only if 
\begin{itemize}
    \item $\limsup c_1(t) c^{-1}_2(t) , \limsup c_2(t) c_1^{-1}(t) <\infty$ and
    \item $\limsup c_1(t) d_{g_t} (\sigma_1(t), \sigma_2(t)), \limsup c_2(t) d_{g_t}(\sigma_1(t), \sigma_2(t)) <\infty$.
\end{itemize}
The preorder of $\mathcal{MBT}_{x_0}$ is defined by 
\begin{equation}
    [(\sigma_1,c_1)]  \geq [(\sigma_2, c_2)] :\Leftrightarrow
    \limsup c_2(t)^{-1} c_1(t)<\infty\  \text{and}\   \limsup c_1(t) d_{g_t} (\sigma_1(t), \sigma_2(t)) <\infty.
\end{equation}
Note that each equivalence class $[(\sigma(t), c(t))]$ gives a unique affine ALE instanton $B$ as the limit of $(X_t, c^2(t)g_t, \sigma(t))$.
Hence an element $[(\sigma(t), c(t))]$ of $\mathcal{MBT}_{x_0}$ is often identified with a bubbling limit $B$ at $x_0$ by denoting $B = [(\sigma(t), c(t))]$ to indicate $\lim (X_t, c^2(t)g_t, \sigma(t)) = B$.

Our main theorem of the present paper is the following.
\begin{Mthm}[Theorem \ref{thm:main theorem}]
    For the given family $(\mathcal{X}, \mathcal{L})\to \Delta$ and a singularity $x_0 \in X_0$, there is a poset isomorphism 
    \begin{equation}
        f: \mathcal{PBT}_{\zeta} \to \mathcal{MBT}_{x_0}
    \end{equation}
    such that $f(\zeta_v) = [(\sigma(t),c(t))] = B$ implies $B \cong Y_{\zeta_v}$ where $Y_{\zeta_v}$ is Kronheimer's (orbifold) ALE gravitational instanton (see \ref{sec:Kronheimer's theory on ALE} for its definition).
    \end{Mthm}
 Note that by Kronheimer's period theory of ALE gravitational instantons, $Y_{\zeta_v}$ are determined explicitly and hence the above theorem gives \textit{complete and explicit} description of bubbling limits of the given family in terms of algebro-geometric data of the given family.
 
 This paper is organized as follows: Section \ref{sec:preliminary} is devoted to preliminaries, convergence notions of K\"ahler manifolds, Kronheimer's theory on ALE hyperk\"ahler gravitational instantons and period theory of polarized K3 surfaces. 
 In section \ref{section:bubbling trees}, constructions of bubbling trees are explained in detail.
In section \ref{sec:main}, the main theorem is proved.
In section \ref{sec:Examples}, some examples and comparison of the main results with local models by de Borbon-Spotti and algebro-geometric models by Odaka are given.
In appendix some remarks on Kronheimer's results on ALE gravitational instantons are presented.

\textbf{Acknowledgement}. The author would like to thank to Yuji Odaka, his supervisor, and Cristiano Spotti, his co-supervisor, for their invaluable advice and
numerous enlightening discussions during the development of this work.
He also would like to thank to Yu-shen Lin, Song Sun and Junsheng Zhang for helpful discussions.
Also he would like to thank M. Enokizono and C. Spotti for letting him know some articles crucial for the present paper.
The author is supported by JSPS KAKENHI Grant Number JP25KJ1453.

\section{Preliminaries}\label{sec:preliminary}
\subsection{Convergence notions in Riemannian Geometry}
In this section, the notions of convergences in Riemannian and K\"ahlerian geometries are introduced.
\begin{dfn}
   Let $A_1, A_2 \subset X$ be compact subsets of a metric space $(X,d)$. 
   Then their Hausdorff distance $d_H (A_1,A_2)$ is given by 
   \begin{equation}
       d_H(A_1,A_2) = \inf \{r \mid \mathrm{B}(A_1;r) \supset A_2, \mathrm{B}(A_2;r) \supset A_1\}.
   \end{equation}

Let $X=(X,d_X), Y=(Y,d_Y)$ be compact metric spaces.
Then their Gromov-Hausdorff distance $d_{GH}(X,Y)$ is given by an infimum of Hausdorff distances among all embeddings $X,Y \hookrightarrow Z$ to some metric spaces $Z$:
\begin{equation}
    d_{GH} (X,Y) = \inf\{ d_H(\varphi_1(X), \varphi_2(Y)) \mid \varphi_j :X,Y \hookrightarrow Z\}.
\end{equation}
     A sequence of metric spaces $\{X_j\}_{j=1}^\infty$ converges to a compact metric space $X$ in the sense of Gromov-Hausdorff if $d_{\mathrm{GH}}(X_j,X) \to 0$.
\end{dfn}
Gromov-Hausdorff convergence of metric spaces is a kind of \textit{uniform convergence} in metric spaces. 
The following notion is a kind of \textit{locally uniform convergence} in metric spaces:
\begin{dfn}
    Let $\{(X_j, d_j, x_j)\}_{j=1}^\infty$ be a sequence of proper metric spaces (i.e. every closed ball of finite diameter is compact) with base points. 
    Then $\{(X_j,d_j,x_j)\}_{j=1}^\infty$ converges to a pointed proper metric space $(X,d,x)$ in the pointed Gromov-Hausdorff sense if $\{(\mathrm{B}_{X_j}(x_j; R),d_j)\}$ converges to $(\mathrm{B}_X(x;R),d)$ in the Gromov-Hausdorff sense with $x_j \to x$ for any $R>0$.
\end{dfn}

The next proposition is a fundamental one to compare geometries of $\{(X_j,d_j)\}_{j=1}^\infty$ and $(X,d)$ (a proof is a simple exercise).
\begin{dfn}
    For a positive number $\varepsilon>0$, an $\varepsilon$-approximating map $f: X \to Y$ between metric spaces is a map (not necessarily continuous!) such that 
    \begin{itemize}
        \item $f(X)$ is $\varepsilon$ dense in $Y$, i.e. $\mathrm{B}(y;\varepsilon) \cap f(X) \neq \varnothing$ for any $y \in Y$ and 
        \item $|d_Y(f(p), f(q)) - d_X(p,q) | \leq \varepsilon$ for any $p,q \in X$.
    \end{itemize}
\end{dfn}
\begin{prop}
    $\{(X_j,d_j)\}_{j=1}^\infty$ converges to $(X,d)$ if and only if for any $\varepsilon>0$ there exists sufficiently large $N$ so that for any $j > N$ we have an $\varepsilon$-approximating map $f_j: X_j \to X$.
\end{prop}
In the present paper, a stronger notion of a convergence, which is more appropriate for K\"ahler geometry and used in the literature (see \cite{DS} , \cite{Song} and their references for example), is used: Let $\{(X_j,g_j)\}_{j=1}^\infty$ be a sequence of K\"ahler manifolds and $(X_\infty, g_\infty)$ be a K\"ahler space (i.e. normal analytic space $X_\infty$ with a K\"ahler metric $g_\infty$ on the regular locus $X^{\mathrm{reg}}$ with bounded potentials around the singular locus).
For simplicity we assume that the metric completion $(\overline{X_\infty^{\mathrm{reg}}}, d_\infty)$ of $(X_\infty^{\mathrm{reg}}, g_\infty)$ is homeomorphic to $X_\infty$.
Then $\{(X_j, g_j)\}_{j=1}^\infty$ converges to $(X_\infty, g_\infty)$ if 
\begin{itemize}
    \item $(X_j, g_j) \to (X_\infty, d_\infty)$ in the Gromov-Hausdorff sense and
    \item for any compact subset $K \subset X_\infty^{\mathrm{reg}}$ there are (not necessarily holomorphic) open embeddings $\iota_j : U \to X_j$, where $U \subset X^{\mathrm{reg}}_\infty$ is an open neighbourhood of $K$, such that 
    \begin{equation}\label{def:definition of convergence of Kahler manifolds}
        \iota^*_j g_j \to g_\infty, \iota_j^* \omega_j \to \omega_\infty
    \end{equation}
    in $C^\infty$ topology on $K$.
\end{itemize}
\begin{rmk}
    The convergence (\ref{def:definition of convergence of Kahler manifolds}) implies that the complex structures on $X_j$ also converges to the complex structure on the limits space $X_\infty$ (as tensors) in the similar manner.
\end{rmk}
\begin{rmk}
    The convergence of K\"ahler forms is automatic in the following sense:
    Let $\{(X_j,g_j)\}_{j=1}^\infty$ be a convergent sequence of Riemannian manifolds.
    Assume that all $(X_j,g_j)$ are K\"ahler.
    Then $\{(X_j,g_j)\}_{j=1}^\infty$ contains a subsequence which is a convergent sequence of K\"ahler manifolds in the sense of (\ref{def:definition of convergence of Kahler manifolds}).
    In fact, let $x \in K \subset X_\infty^{\mathrm{reg}}$ be a smooth point and consider $\{\iota_j ^*\omega_{j,x} \}_{j=1}^\infty \subset \wedge^2 T^*_x X_\infty$.
    This is a bounded set with respect to $\| \cdot\|_{g_\infty}$ since $\|\omega_j\|_{g_j} = 1$ and $\iota_j^*g_j \to g_\infty$.
    Hence there is a convergent subsequence $\{\omega_{j_k,x}\}_{k=1}^\infty$ converging to $\omega_{\infty,x} \in \wedge^2 T^*_x X_\infty$ .
    A standard argument using the parallel transport of $g_\infty$ implies $\omega_{\infty,x}$ defines a K\"ahler form $\omega_\infty$ on $X_\infty^{\mathrm{reg}}$ as desired.
\end{rmk}
\begin{rmk}
    The assumption $X_\infty$ to be a normal analytic space is a natural one at least under \textit{non-collapsing} condition.
    In fact, if $\{(X_j,g_j)\}$ are projective K\"ahler-Einstein manifolds with uniformly bounded diameters from above and uniformly bounded volumes from below,
    then the limit space is a projective K\"ahler space at worst klt singularities \cite{DS}.
    
\end{rmk}

\subsection{Kronheimer's theory on ALE hyperk\"ahler gravitational instantons}\label{sec:Kronheimer's theory on ALE}
In this section, Kronheimer's theory on ALE hyperk\"ahler gravitational instantons, which appear as bubbling limits of non-collapsing Gromov-Hausdorff limits of K3 surfaces, is reviewed by following the original papers \cite{Kro1} \cite{Kro2}.

\begin{dfn}
    An ALE hyperk\"ahler gravitational instanton $Y = (Y,g,I,J,K)$ is a non-compact and complete hyperk\"ahler $4$-manifold such that there is a diffeomorphism
    \begin{equation}\label{eq:coordinate at infinity}
    \varphi : (\C^2 \setminus \mathrm{B}(0;R))/\Gamma \to Y\setminus K
    \end{equation}
    for some compact subset $K \subset Y$, $R >0$ and a finite subgroup $\Gamma \subset \mathrm{SL}(2;\C)$ such that 
    \begin{equation}\label{cond:asymptotic of ALE at infinity}
    \left\{
    \begin{aligned}
        \|\nabla^k(g^\mathrm{Euc} - \varphi^*g) \|_{g^{\mathrm{Euc}}} &= O(r^{-4-k}), \\
        \|\nabla^k( \omega_I^{\mathrm{Euc}} - \varphi^*\omega_I) \|_{g^\mathrm{Euc}} &= O(r^{-4-k}), \\
        \|\nabla^k( \omega_J^{\mathrm{Euc}} - \varphi^*\omega_J) \|_{g^\mathrm{Euc}} &= O(r^{-4-k}),\ \ \\
        \|\nabla^k( \omega_K^{\mathrm{Euc}} - \varphi^*\omega_K) \|_{g^\mathrm{Euc}} &= O(r^{-4-k}),
        \end{aligned}\right.
    \end{equation}
    as $r = \|z\|\to \infty$ for any positive integer $k$. 
\end{dfn}
Note that there is a canonical hyperk\"ahler structure on $\C^2$ via the identification $\mathbb{H} \cong \C^2$ with Hamilton's quaternion number $\mathbb{H}$.
The diffeomorphism in (\ref{eq:coordinate at infinity}) is called a \textit{coordinate at infinity} and $Y$ is called an ALE hyperk\"ahler gravitational instanton of type $\Gamma$.
The main theorems of Kronheimer\cite{Kro1}\cite{Kro2} are the following:

\begin{thm}\label{thm:Hyperkahler structure on the minimal resolution of Kleinnian singularity}

    Let $\Gamma \subset \mathrm{SL}(2;\C)$ be a finite subgroup and let $Y$ be the underlying differentiable $4$-manifold of the minimal resolution of $\C^2/\Gamma$.
    For a triple $\bm{\kappa} = (\kappa_1,\kappa_2, \kappa_3) \in H^2(Y;\R)^3$, consider the following condition:
    \begin{equation}\label{cond:generality condition}
        \mathrm{For\ any}\ \theta \in H_2(Y,\Z)\ \mathrm{with}\ \theta\cdot\theta =-2,\ \langle\theta,\kappa_i\rangle \neq 0\ \mathrm{at\ least\ one}\ \kappa_i.
    \end{equation}
    Then for any triple $\bm{\kappa} \in H^2(Y; \R)^3$ satisfying the condition (\ref{cond:generality condition}), there exists a unique hyperk\"ahler structure $(g_{\bm{\kappa}}, I_{\bm{\kappa}}, J_{\bm{\kappa}}, K_{\bm{\kappa}})$ on $Y$ such that $(Y,g_{\bm{\kappa}},I_{\bm{\kappa}}, J_{\bm{\kappa}}, K_{\bm{\kappa}})$ is an ALE hyperk\"ahler gravitational instanton and the triple $ ([\omega_{I_{\bm{\kappa}}}], [\omega_{J_{\bm{\kappa}}}], [\omega_{K_{\bm{\kappa}}}])$ is equal to the given $\bm{\kappa}$.
\end{thm}

\begin{thm}\label{thm:surjectivity of the period mapping of ALE}
    Let  $(Y,g,I,J,K)$ be an ALE hyperk\"ahler gravitational instanton of type $\Gamma$. Then $Y$ is diffeomorphic to the minimal resolution of $\C^2/\Gamma$ and the triple of cohomology classes of its K\"ahler forms $([\omega_I], [\omega_J], [\omega_K])$ satisfies the condition (\ref{cond:generality condition}) (under the diffeomorphism).
\end{thm}

\begin{thm}\label{thm:torelli type theorem of ALE}
    Let $Y_1$ and $Y_2$ be ALE hyperk\"ahler gravitational instantons of type $\Gamma$.
    Assume there exists a diffeomorphism $Y_1 \to Y_2$ which preserves the triples of the cohomology classes of their K\"ahler forms.
    Then $Y_1$ and $Y_2$ are isomorphic as hyperk\"ahler manifolds.
    In particular, the hyperk\"ahler structures $(g_{\bm{\kappa}}, I_{\bm{\kappa}}, J_{\bm{\kappa}}, K_{\bm{\kappa}})$ and $(g_{\alpha({\bm{\kappa}})}, I_{\alpha(\bm{\kappa})}, J_{\alpha(\bm{\kappa})}, K_{\alpha(\bm{\kappa})})$, defined in Theorem \ref{thm:Hyperkahler structure on the minimal resolution of Kleinnian singularity}, are isomorphic for an isometry $\alpha \in \mathrm{O}(H^2(Y;\Z))$.
\end{thm}
Here a \textit{hyperk\"ahler structure} on $Y$ means a quadruple $(g,I,J,K)$ and an \textit{isomorphism of hyperk\"ahler structures} is an isometry $(Y_1,g_1) \to (Y_2,g_2)$ which preserves all complex structures (hence, in particular, it preserves K\"ahler forms).

The above theorems can be generalized for orbifolds and play an important role in the proof of the main theorem.
For simplicity, $Y$ is regarded as a complex manifold by the complex structure $I$ and assume $\kappa_1 =0$, which is equivalent to assume $Y$ is affine (see Lemma 3.9. \cite{Kro1}), in the rest of this section.

\begin{dfn}
    An affine ALE gravitational instanton with ADE singularities is an orbifold hyperk\"ahler $4$-manifold $Y = (Y,g,I,J,K)$ such that 
    \begin{itemize}
        \item $(Y,I)$ is an affine variety at worst ADE-type singularities such that its minimal resolution is a deformation $\C^2/\Gamma$ and 
        \item there is a coordinate at infinity with the asymptotics (\ref{cond:asymptotic of ALE at infinity}).
        
    \end{itemize}
        
\end{dfn}
Then the following hold.
\begin{thm}\label{thm:Hyperkahler structure on deformations on Kleinian singularity}
   Fix the underlying differentiable manifold of $\widetilde{\C^2/\Gamma}$, say $Y$.
   For a cohomology class $\kappa \in H^2(Y;\C)$, let $R(\kappa) \subset H^2(Y;\Z)$ be the set of roots annihilated by $\kappa$:
   \begin{equation}
       R(\kappa) :=\{\theta \in H^2(Y,\Z)  \mid \theta^2 = -2 , \langle \kappa, \theta \rangle =0 \}.
   \end{equation}
   Then there exists a unique affine ALE gravitational instanton $Y_{\kappa}=(Y_{\kappa}, g_{\kappa}, I_{\kappa}, J_{\kappa}, K_{\kappa})$ 
   such that
   \begin{itemize}
       \item $[\rho^*(\omega_{J_\kappa} + \sqrt{-1} \omega_{K_\kappa}) ] = \kappa$ where $\rho : Y \to Y_{\kappa}$ is the minimal resolution and 
       \item $R(\kappa) = \{\theta \in \mathrm{Span}_{\Z} \{[E_j ] \}_{j=1}^n \mid \theta^2 = -2 \}$ where $E= \bigcup_{j=1}^n E_j$ is the exceptional divisors of the resolution.
   \end{itemize}
\end{thm}
\begin{thm}\label{thm:orbifold surjectivity of torelli of ALE}
    Let $(Y, g, I,J,K)$ be an affine ALE gravitational instanton.
   Consider the nowhere vanishing holomorphic 2 form $\Omega := \omega_J + \sqrt{-1} \omega_K$.
    Then there exists a cohomology class $\kappa \in H^2(\widetilde{\C^2/\Gamma}; \Z)$ such that $Y\cong Y_{\kappa}$ as K\"ahler surfaces and $\Omega$ coincides with $\Omega_{\kappa} = \omega_{J_\kappa} + \sqrt{-1} \omega_{K_\kappa}$ under the isomorphism.
    \end{thm}
    \begin{thm}\label{thm:orbifold injectivity of period of ALE}
    Let $Y_1$ and $Y_2$ be affine ALE gravitational instantons.
    Assume $\kappa_1$ and $\kappa_2$ corresponds to $Y_1$ and $Y_2$ respectively, i.e. $Y_j \cong Y_{\kappa_j}$.
    If there are a constant $c \in \C\setminus\{0\}$ and an automorphism $\alpha \in \mathrm{O}(H^2(\widetilde{\C^2/\Gamma};\Z))$ so that $\kappa_1 = c \alpha(\kappa_2)$ then $Y_1 \cong Y_2$.
    \end{thm}
    \begin{rmk}
        These theorems are essentially shown in \cite{Kro1}\cite{Kro2} (or straight forward generalizations of the original proofs).
        Hence the author believes that they are well-known to the experts.
        However he can not find their proofs in the literature. 
        So we give some explanations about the theorems and sketches of proofs in Appendix A for reader's convenience.
    \end{rmk}
    
\subsubsection{Period theory of ALE hyperk\"ahler gravitational instantons}
\begin{dfn}[Period mapping of ALE hyperk\"ahler gravitational instantons]
    Let $Y = (Y,g,I,J,K)$ be an ALE hyperk\"ahler gravitational instanton of type $\Gamma$.
    A marking of $Y$ is an isomorphism $\alpha$ of lattices
    \begin{equation}
        \alpha : H^2(Y;\Z) \to \h_\Z,
    \end{equation}
    where $\h_\Z$ is the root lattice of a Cartan sub-algebra of the complex simple Lie algebra corresponding of type $\Gamma$.
    A marked ALE hyperk\"ahler gravitational instanton is a pair of ALE hyperk\"ahler gravitational instanton and its marking.
    An isomorphism of marked ALE hyperk\"ahler gravitational instantons is an isomorphism of hyperk\"ahler manifolds compatible with the markings.
\end{dfn}

\begin{dfn}
    Let $Y$ be an affine ALE gravitational instanton of type $\Gamma$.
    Let $\widetilde{Y}$ be the minimal resolution.
    Define a sub-space $IH^2(Y) \subset H^2 (\widetilde{Y};\Z)$ by 
    \begin{equation}
        IH^2(Y) := \{ x \in H^2(\widetilde{Y};\Z) \mid \langle x,\theta  \rangle =0 , \forall \theta \in R(Y) \},
    \end{equation}
    where 
    \begin{equation}
        R(Y) = \left\{ \theta = \sum_{j=1}^n c_j [E_j] \in H^2( \widetilde{Y}; \Z) \ \middle| \ \theta^2 =-2 \right\}
    \end{equation}
    for the exceptional divisors $E_j \subset \widetilde{Y}$ of the resolution.
    Then a marking $\alpha$ of $Y$ is an embedding $IH^2(Y) \hookrightarrow \h_\Z$ which is a restriction of an isomorphism $\tilde{\alpha} : H^2(\widetilde{Y}; \Z) \to \h_\Z$.
    A marked affine ALE gravitational instanton and its morphism are defined in a similar manner as above.
\end{dfn}
The theorems in the previous theorems implies that
\begin{equation}
    \h_\C^\circ := \h_\C \setminus \bigcup_{\theta^2 =-2} H_{\theta}, 
\end{equation}
where $H_\theta = \theta^\perp \subset \h_\C$, parametrizes the isomorphism classes of marked ALE gravitational instantons and $\h_\C$ does the isomorphism classes of marked affine ALE gravitational instantons. 
Next, the completion of $\h_\C^\circ$ along $H_\theta$ is presented (the author believes that this is well-known for experts but he can not find the literature. So we give proofs here for reader's convenience).
A key ingredient is that for a convergent sequence in the period domain, the corresponding sequence of ALE instantons is convergent with respect to the Gromov-Hausdorff topology.
The next proposition follows easily from the original construction in \cite{Kro1} (see the proof of Lemma 3.3.).
\begin{prop}\label{prop:ALE parameterized on the discriminant loci}
    Take $\zeta \in H_\theta  \subset \h_\C$ for some roots $\theta$, and consider a sequence $\{\zeta_j\}_{j=1}^\infty \subset \h_\C^o$ which converges to $\zeta$. Then $\{Y_{\zeta_j}\}_{j=1}^\infty$ converges to $Y_{\zeta}$ in the sense of Cheeger-Gromov, i.e. for any compact subset $K \subset Y^\mathrm{reg}_\infty$ there are open embeddings 
    \begin{equation*}
        \iota_j : U \to Y_{\zeta_j},
    \end{equation*}
    defined on an open neighborhood $U \subset Y^\mathrm{reg}_\infty$ of $K$,
    such that $\iota_j^*\omega_{I_{\zeta_j}} \to \omega_{I_\zeta}$ in $C^\infty(K)$ topology (and the same for $\omega_J$ and $\omega_K$).
\end{prop}
\begin{prop}
    In the above situation, $Y_{\zeta_j} \to Y_{\zeta}$ in the sense of (pointed) Gromov-Hausdorff (with suitable base points).
\end{prop}

\begin{proof}
 The idea of the proof is the same as that of proposition 6.5. of \cite{OO}.
It is sufficient to show that 
\begin{equation} 
\mathrm{B}(y_\infty; r) \to \mathrm{B}(y_j; r)
\end{equation}
for sufficiently small $r > 0$.
Fix an $r >0$ and a quotient map
\begin{equation}
\begin{aligned}
   \pi : \mathrm{B}(0; r) &\to Y_\zeta \\
   0 &\mapsto y_\infty
   \end{aligned}
\end{equation}
such that $\pi^*g_{\zeta} = g^\mathrm{Euc} + (\text{h.o.t.})$.
Assume $r>0$ is sufficiently small so that $y_\infty$ is the only one singularity contained in $\pi(\mathrm{B}(0;r))$.
Let $K = \pi( \{ z \in \mathrm{B}(0;r) \mid \varepsilon \leq \|z\| \leq r_0 -\varepsilon \}) \subset Y_\zeta^\mathrm{reg}$ and take open embeddings
\begin{equation}
    \iota_j : U \to Y_{\zeta_j}
\end{equation}
which give the convergence $\iota_j ^* g_{\zeta_j} \to g_\zeta$.
Define a map $f_j : \pi(\mathrm{B}(0; r)) \to Y_{\zeta_j}$ by 
\begin{equation}
    \begin{aligned}
         f_j : \pi(\mathrm{B}(0; r)) &\to Y_{\zeta_j} \\
         p &\mapsto 
        \left\{ \begin{array}{cc}
         \iota(p) &\text{ if } p \in U\\   
            y_j  & \text{otherwise},
         \end{array}\right. 
         \end{aligned}
\end{equation}
where $y_j \in Y_{\zeta_j}$ is a point in the relatively compact component of $Y_{\zeta_j} \setminus \iota_j(U)$.
The map $f_j$ is an approximation map. 
In fact, by the Bishop-Gromov inequality, we have 
\begin{equation}\label{ineq:Bishop Gromov}
    1 \geq \frac{\mathrm{Vol}_{g_{\zeta_j}}(\mathrm{B}(x;\rho))}{C \rho^4} \geq \frac{1}{|\Gamma|}
\end{equation}
for any $x $, a point in the relatively compact component of $Y_{\zeta_j} \setminus \iota_j({U})$, $ 0<\rho < r$ and sufficiently large $j$ with some positive constant $C$ independent of $x, \rho,j$.
If we have a point $x_j$ in the relatively compact component of $Y_{\zeta_j} \setminus \iota_j(U)$ such that $d_{g_{\zeta_j}}(x_j, \iota_j(U)):= r_j > \varepsilon$,
then the inequality (\ref{ineq:Bishop Gromov}) implies 
\begin{equation}
   \varepsilon^4 \geq \mathrm{Vol}_{g_{\zeta_j}}(\mathrm{B}(x_j; \varepsilon)) \geq \mathrm{Vol}_{g_{\zeta_j}}(\mathrm{B}(x_j; r_j)) > C \rho_j^4,
\end{equation}
for a constant $C >0$ independent of $j, x, \varepsilon$.
Then $f_j$ are $\varepsilon$-approximation maps for sufficiently large $j$.
\end{proof}

The rest of this subsection is devoted to review an algebro-geometric aspect of ALE hyperk\"ahler gravitational instantons.
Identify $\h_\R^{\oplus 3} \cong \h_\R \oplus \h_\C$ and regard $Y_\zeta$ as a complex manifold with a complex structure $I_\zeta$. 
 Then we have the following properties.
 \begin{prop}[Lemma 3.9. and Proposition 3.10. \cite{Kro1}]
     If $\zeta=(\zeta_r, \zeta_c) =(0, \zeta_c)$, then $Y_\zeta$ is an affine variety.
     Furthermore, there is a natural map 
     \begin{equation}
         Y_{(\zeta_r, \zeta_c)} \to Y_{(0,\zeta_c)}
     \end{equation}
     which is a proper birational map. 
     In particular, if we take $\zeta_r$ so that $\zeta \in (\h_\R \oplus \h_\C)^o$, this gives the minimal resolution of the singularities.
 \end{prop}
 It is straightforward that the exceptional sets of the birational map 
 \begin{equation}
     Y_{(\zeta_r, \zeta_c)} \to Y_{(0,\zeta_c)}
 \end{equation}
 corresponds to roots $\theta$ which are perpendicular to $\zeta_c$ but not to $\zeta_r$.
 \begin{prop}[section 4. \cite{Kro1}]
     Let $\zeta: \Delta \to \h_\C$ be a holomorphic function over the unit disc $\Delta \subset \C$. 
     Then $\mathcal{Y}_\zeta := \bigcup_{t\in \Delta} Y_{(0,\zeta(t))}$ admits a natural complex structure such that $\mathcal{Y} \to \Delta$ is a flat family over $\Delta$. In particular, if we take $\zeta_r \in \h_\R $ so that $\tilde\zeta(t)= (\zeta_r, \zeta(t)) \in (\h_\R \oplus \h_\C)^o$ for any $t \in \Delta$, we have a simultaneous resolution 
    \begin{equation}
        \mathcal{Y}_{(\zeta_r, \zeta(t))} \to \mathcal{Y}_{(0,\zeta(t))}.
    \end{equation}
 \end{prop}
 Furthermore, a quotient space $\h_\C /W$ by the Weyl group $W$ is the Kuranishi space of the singularity $\C^2/\Gamma$.

\subsection{The moduli spaces and its metric completion of K\"ahler K3 surfaces}
In this section, the period theory and its metric completion of K3 surfaces are reviewed.
A main reference of this subsection is \cite{KT}.
\begin{dfn}
    A compact non-singular complex surface $X$ is a K3 surface if it has a nowhere vanishing holomorphic $2$-form on $X$ and simply connected.
    A compact complex surface with ADE-singularities $X$ is a K3 surface with ADE-singularities if its minimal resolution is a K3 surface.
\end{dfn}
It is well-known that K3 surfaces are all diffeomorphic to each other hence we call the underlying differentiable $4$-manifold the K3 manifold.
In particular, their lattices $H^2(X;\Z)$ are all isometric each other.
The lattice is the unimodular even lattice of signature $(3,19)$ and hence it is isometric to $E_8^{\oplus 2} \oplus U^{\oplus 3}$, where $U$ is the hyperbolic lattice of rank 2:
\begin{equation}
    U = 
    \begin{bmatrix}
        0 &1\\
        1 &0
    \end{bmatrix}.
\end{equation}
The lattice is called the \textit{K3 lattice} and is denoted by $L$.
For each positive integer $d$, a primitive element $\lambda \in L$ with $\langle \lambda,\lambda\rangle= 2d$ is unique up to the $\mathrm{O}(L)$ action.
Fix such element $\lambda$, say $\lambda =e + d f$ for a basis of one of $U$ with $\langle e,e \rangle = \langle f,f \rangle =0, \langle e,f\rangle =1$, and let $L_{2d} = \lambda^\perp \subset L$.
A marking $\alpha: H^2(X,\Z) \to L$ of a K3 surface $X$ is an isomorphism of lattices and a pair $(X, \alpha)$ of K3 surface $X$ and its marking $\alpha$ is called a marked K3 surface.
Marked K3 surfaces $(X_1, \alpha_1)$ and $(X_2,\alpha_2)$ are isomorphic if there is an isomorphism $f : X_1 \to X_2$ such that $\alpha_2 = \alpha_1 \circ f^*$.
\begin{dfn}
    Let $\mathcal{M}$ be the isomorphic classes of marked K3 surfaces.
    The period mapping $P: \mathcal{M} \to \mathbb{P}(L\otimes \C)$ is a map defined as follows: For a marked $K3$ surface $(X,\alpha)$, take a nowhere vanishing holomorphic $2$-form $\Omega$ on $X$.
    Then $P(X,\alpha) = [\alpha([\Omega])] \in \mathbb{P}(L\otimes \C)$.
\end{dfn}
Note that the period mapping is independent of the choice of the holomorphic $2$ form $\Omega$ since a holomorphic $2$ form on a K3 surface is unique up to constant as a K3 surface is compact.
The period map satisfies the following.
\begin{thm}
    Let $\widetilde{\Omega}_{K3} \subset L\otimes \C$ be a subset defined by
    \begin{equation}
      \widetilde{\Omega}_{K3} =  \{ x \in L\otimes \C \mid \langle x, x\rangle =0, \langle x, \overline{x}\rangle >0 \}.
    \end{equation}
    Then the period map defines a surjection $\mathcal{P}: \mathcal{M} \to \Omega_{K3} := \mathbb{P}(\widetilde{\Omega}_{K3}) $ which induces a bijection 
    \begin{equation}
        (\text{Isomorphism classes of K3 surfaces}) \to \Omega_{K3}/\mathrm{O}(L).
    \end{equation}
    In particular, $\mathcal{F} := \Omega_{K3}/\mathrm{O}(L)$ is the moduli space of K3 surfaces.
\end{thm}
Let $M$ be an ample line bundle on a K3 surface $X$.
Then for a marking $\alpha$, $\lambda := \alpha (c_1(M))\in L$ is a vector with length $\lambda^2 = \mathrm{deg}(M)$.
In particular, if we take $M$ to be a primitive, i.e. there are no line bundle $N$ on $X$ so that $M = N^{\otimes n}$ for some $n>1$, and degree $2d$, the period of $X$ takes its value in $\mathbb{P}(L_{2d} \otimes \C)$.

However, the period map does not give a bijection from $\mathcal{M}_{2d}$, the isomorphism classes of primitively polarized marked K3 surfaces, to $\Omega_{2d} :=\mathbb{P}(\widetilde{\Omega}_{2d})$, where $\widetilde{\Omega}_{2d} := \widetilde{\Omega}_{K3} \cap L_{2d} \otimes \C$.
Some loci in $\Omega_{2d}$, so called \textit{discriminant loci}, $H = \bigcup H_{\theta}$ have to be removed.
They are defined as
\begin{equation}
    H_\theta := \{ x \in \Omega_{2d} \mid \langle x, \theta\rangle =0 \}
\end{equation}
where the union run through all roots $\theta \in L_{2d}$. 
The loci can be filled by putting K3 surfaces with ADE-singularities, i.e. the period map gives bijection from the isomorphism classes of marked K3 surfaces with ADE singularities equipped with a primitive ample line bundle of degree $2d$ to $\Omega_{2d}$ which sends a K3 surface with singularities to a discriminant locus.
\begin{thm}
    Let $\mathcal{M}_{2d}^o$ be the set of isomorphism classes of marked primitively polarized K3 surfaces of degree $2d$ and let $\mathcal{M}_{2d}$ be the set of isomorphism classes of marked primitively polarized K3 surfaces with ADE singularities of degree $2d$
    Then the period map gives bijection 
    \begin{equation}
    \begin{aligned}
        \mathcal{M}^o_{2d} &\to \Omega_{2d} \setminus H\\
        \mathcal{M}_{2d} & \to \Omega_{2d}.
    \end{aligned}
    \end{equation}
    Furthermore, $\mathcal{F}^o_{2d} = (\Omega_{2d} \setminus H)/\Gamma$ and $\mathcal{F}_{2d} = \Omega_{2d}/\Gamma$ gives the moduli spaces, where 
    \begin{equation}
        \Gamma = \{\alpha \in \mathrm{O}(L) \mid \alpha \lambda = \lambda \}.
    \end{equation}
\end{thm}

The image of $H_\theta$ by the quotient map $\Omega_{2d} \to \mathcal{F}_{2d}$ is denoted by $\mathcal{H}_\theta$ and the union $\mathcal{H} = \cup_{\theta^2 =-2} \mathcal{H}_\theta$ defines a divisor of $\mathcal{F}_{2d}$, the so called \textit{Heegner divisor}. As in the case of ALE gravitational instantons, a metric picture of this completion of discriminant loci is also valid. 
By Yau's existence theorem, there exists a Ricci-flat K\"ahler metric on a K3 surface for every K\"ahler class uniquely.
Yau's existence theorem for K3 surfaces with ADE singularities in the category of orbifold K\"ahler metrics also holds.
Hence a K3 surface with ADE-singularities equipped with an ample line bundle can be identified with an orbifold Hodge manifold with a Ricci flat metric.
For a holomorphic curve $s: \Delta \to \mathcal{F}_{2d}$, there exists a flat family 
\begin{equation}
\mathcal{X}_s \to \Delta
\end{equation}
of K3 surfaces over the unit disc $\Delta \in \C$ with Ricci flat K\"ahler (possibly orbifold) metrics.
Then it converges to the central fiber as $t\to 0$ in the sense of Gromov-Hausdorff and Cheeger-Gromov \cite{KT} (see also Proposition 6.7. \cite{OO} and \cite{Anderson}).

\subsection{Singularities and period mapping of ADE-singularities}
Let $\h$ be the root lattice of a Cartan sub-algebra corresponds to a finite subgroup $\Gamma \subset \mathrm{SL}(2;\C)$.
Let $\zeta \in \h_\C$ be a vector perpendicular to some roots and let $\{\theta_\alpha\}_{\alpha=1}^n$ be the set of positive roots perpendicular to $\zeta$.
Then, the ALE hyperk\"ahler gravitational instanton $Y_\zeta$ has singularities which is again of ADE-type (with smaller Milnor numbers).
We can identify the singularities in terms of $\{\theta_\alpha\}_{\alpha=1}^n$ as follows.
\begin{lem}\label{lem:identification of the singularities by lattice data}
    In the above situation, decompose $\{\theta_\alpha\}_{\alpha =1}^n = \{\theta_{1,\alpha}\}_{\alpha=1}^{n_1}\sqcup \cdots \sqcup\{\theta_{k,\alpha}\}_{\alpha=1}^{n_k}$ so that each of them forms a maximal irreducible root systems (with their negatives).
    Then the singularities of $Y_\zeta$ is equal to $\{x_1, \ldots, x_k\}$ such that $x_j$ is the ADE singularity isomorphic to $0 \in \C^2/\Gamma_j$ where $\Gamma_j$ is the finite subgroup of $\mathrm{SL}(2;\C)$ corresponding to the root system $\{\theta_{\alpha_j}\}_{\alpha=1}^{n_k}$.
\end{lem}
\begin{proof}
   It is a special case of a general theory of equisingular deformations of rational double points. For instance, see Lemma 6.6. in \cite{Wahl}.
    
\end{proof}

\section{Bubbling trees of polarized K3 surfaces}\label{section:bubbling trees}
\subsection{A localization of periods of K3 surfaces to ALE gravitational instantons}\label{sec:localization of period}
In this subsection, the period theories of K3 surfaces and ALE hyperk\"ahler gravitational instantons are related in an explicit way.
Let
\begin{equation}
(\mathcal{X},\mathcal{L}) \to \Delta
\end{equation}
be a family of polarized K3 surfaces over the unit disc $\Delta \subset \C$.
Assume that a fiber $X_t$ on $t \neq 0$ is smooth and the central fiber $X_0$ has ADE singularities.
If $L_t := \mathcal{L} \mid_{X_t}$ is primitive (i.e. there is no line bundle $M$ such that $L = M^{\otimes m}$ for some $m >1$) and has degree $2d$ then there exists a holomorphic curve
\begin{equation}
    \Phi: \Delta \to \mathcal{F}_{2d}
\end{equation}
from $\Delta$ to the moduli space $\mathcal{F}_{2d}$ of primitively polarized K3 surfaces of degree $2d$.
Note that $\Phi(0) \in \mathcal{H}_{\theta}$ for some $\theta$ by the assumption, where $\mathcal{H}_\theta$ is a Heegner divisor corresponding to a hyperplane $H_\theta \subset L_{2d}\otimes \C$ cut out by a root $\theta \in L_{2d}$ (recall that $L_{2d}$ is the orthogonal complement of a primitive root of length $2d$).
After a suitable base change (if it is necessary), fix a minimal simultaneous resolution 
\begin{equation}
    \widetilde{\mathcal{X}} \to \mathcal{X}
\end{equation}
of the family.
Let $\widetilde{\mathcal{L}}$ be the pullback of $\mathcal{L}$ on $\widetilde{\mathcal{X}}$.
Let $\omega_t \in c_1(L_t)$ be the Ricci-flat K\"aher metric on $X_t$ and take a holomorphic volume form $\Omega_t$ on $X_t$ so that 
\begin{equation}\label{eq:Normalization of holomorphic 2forms}
    \omega_t^2 = \frac{1}{2} \Omega_t \wedge \overline{\Omega}_t.
\end{equation}
As $\widetilde{\mathcal{X}} \to \Delta$ is a smooth family of manifolds, (by replacing $\Delta$ sufficiently smaller if it is necessary) it can be trivialized as a family of $C^\infty$ manifolds:
\begin{equation}
    \widetilde{\mathcal{X}} \cong_{\mathrm{diffeo.}} X \times \Delta.
\end{equation}
For a fixed marking 
\begin{equation}
    \alpha : H^2(X;\Z) \to L,
\end{equation}
there is a holomorphic curve 
\begin{equation}
\begin{aligned}
    \mathcal{P} :\Delta &\to L_{2d} \otimes \C\\
    t &\mapsto \alpha([\Omega_t]).
\end{aligned}
\end{equation}
Assume that the holomorphic $2$-forms $\Omega_t$ with the normalization (\ref{eq:Normalization of holomorphic 2forms}) are taken so that $P$ is a lifting of $\Phi$:
\begin{equation}
    \begin{tikzcd}
        \Delta \arrow[r,"\mathcal{P}"] \arrow[d] & \widetilde{\Omega}_{2d}\arrow[d] \subset L_{2d}\otimes \C  \\
        \Delta \arrow[r, "\Phi"] & \mathcal{F}_{2d}.
    \end{tikzcd}
\end{equation}

Let $x_0 \in X_0 \cong 0 \in \C^2/\Gamma$ be a singularity and let $E_1,\ldots, E_n$ be the irreducible components of the exceptional divisor of the minimal resolution.
Then, for $\theta_j := \alpha([E_j])$, $\mathcal{P}(0) \in H_{\theta_j}$, or equivalently, $\Phi(0) \in \mathcal{H}_{\theta_j}$.
Note that a sub-lattice 
\begin{equation}
\h_\Z := \mathrm{Span}_{\Z} \{\theta_1,\ldots,\theta_n\} \subset L_{2d}
\end{equation}
is isometric to the root lattice of a Cartan sub-algebra of a simple complex Lie algebra of the same ADE-type with $x_0 \in X_0 \cong 0 \in \C^2/\Gamma$.
Take a suitable open neighborhood 
\begin{equation*}
   x_0 \in \mathcal{U} \subset \mathcal{X}
\end{equation*}
which is a deformation of the singularity $x_0 \in X_0 \cong 0\in \C^2/\Gamma$.
The inclusion $\mathcal{U} \hookrightarrow \mathcal{X}$ induces a minimal simultaneous resolution
\begin{equation}
    \widetilde{\mathcal{U}} \to \mathcal{U}
\end{equation}
and the trivialization $\widetilde{\mathcal{X}} \cong X \times \Delta$ induces a trivialization
\begin{equation}
    \widetilde{\mathcal{U}} \cong_{\text{diffeo.}} U \times \Delta.
\end{equation}
Note that $H^2(U,\Z)$ is spanned by $\{[E_j]\}_{j=1}^n$ and then we have an isomorphism
\begin{equation}
    H^2(U,\Z) \cong \h_\Z \subset L_{2d}.
\end{equation}
In particular, we have a holomorphic map 
\begin{equation}
    \zeta: \Delta \to \h_\C
\end{equation}
by sending $[\Omega_t\mid _{U_t}]$ by the above isomorphism.
Note that $[\omega_t]$ is mapped to $0$ identically by the above isomorphisms (recall that we are working on polarized case) hence the $\h_\R$ component of the above map is trivial.
Then it is straightforward 
\begin{equation}
    \zeta = \pi_\h \circ \mathcal{P}
\end{equation}
where $\pi_\h$ is the orthogonal projection to $\h_\C$.
\begin{dfn}
    Let $\mathcal{P} : \Delta \to L_{2d}\otimes \C$ be a holomorphic map and let $D \subset L_{2d}\otimes \C$ be a definite subspace. Then a \textit{localization of $\mathcal{P}$ along $D$} is a holomorphic map $\pi_{D} \circ \mathcal{P} : \Delta \to D$ where $\pi_{D}$ is the orthogonal projection to $D$.
\end{dfn}
In the above situation, $\zeta $ is the localization of $\mathcal{P}$ along $\h_\C$.

\subsection{Bubbling trees from periods}

The rest of this section is devoted to give a full description of bubbling trees of non-collapsing limits of ALE hyperk\"ahler gravitational instantons using the previous results.
\begin{dfn}\label{dfn:pbt}
    A period-bubbling tree $\mathcal{PBT}_{\zeta}$ of a holomorphic map $\zeta : \Delta \to \h_\C$ is a partially ordered set of pairs $\{([\zeta_v], \mathbb{P}(\h_v))\}$ of linear subspaces $\mathbb{P}(\h_v)$ with $[\zeta_v] \in \mathbb{P}(\h_v)$ defined as follows:
\begin{itemize}
    \item the root is $([\zeta_r], \mathbb{P}(\h_\C))$ if $\zeta (t) = t^k \zeta_r + O(t^{k+1})$ with $\zeta_r \neq0$
    \item a vertex $([\zeta_w], \mathbb{P}(\h_w)) $ is a child of $([\zeta_v], \mathbb{P}(\h_v))$ if and only if $\pi_{\zeta_{v}} \circ \zeta(t) = t^l \zeta_w + O(t^{l+1})$ where $\pi_{\zeta_{v}}$ is the projection to $\h_v$ and $\h_w$ is spanned by a maximal irreducible root system $\{\theta_j\}_{j=1}^{n_w} \subset \h_v$ perpendicular to $\zeta_v$.
\end{itemize}
\end{dfn}
Here a finite poset $T$ is a tree if there is a unique maximal element and $u,v \geq w$ implies either $u \geq v$ or $v \geq u$.
The maximal element of $T$ is called the \textit{root} of $T$.
A vertex $v$ is the parent of $w$ if $v = \min \{u \in T \mid u > w\}$ and $w$ is a child of $v$ if $v$ is the parent of $w$.
Note that a parent is unique but a child is not unique.
A vertex $v$ is a leaf of $T$ if there are no children of $v$.

\subsection{Bubbling trees from Gromov-Hausdorff limits}\label{sec:mbt}
In this subsection, another notion of bubbling trees is introduced for a limit space $(X_\infty, g_\infty)$ of a sequence of K\"ahler manifolds $\{(X_j, g_j)\}$ (see \cite{dBS} for an example of a bubbling tree used to study degenerations of K\"ahler-Einstein metrics).
Let $\{(X_j,g_j)\}$ be a sequence of K\"ahler manifolds converging to a normal K\"ahler space $(X_\infty,g_\infty)$.
For simplicity, assume that the convergence is non-collapsing.
To define the bubbling tree at $x_\infty$, some notions on bubbling limits are introduced by following \cite{Song}.
\begin{dfn}
 Let $(X,g)$ be a (possibly non compact) K\"ahler manifold. 
    Its tangent cone at $p$, denoted by $C_p(X)$, is a pointed Gromov-Hausdorff limit 
    \begin{equation}
        \lim_{r \to \infty} (X, r^2 g, p).
    \end{equation}
    An asymptotic cone $C_\infty(X)$ is a pointed Gromov-Hausdorff limit 
    \begin{equation}
        \lim_{r \downarrow 0} (X,r^2 g, p)
    \end{equation}
    for some $p \in X$ (it is independent of the choice of $p$).
    A cone $C$ is called a tangent cone of $B$ at $p \in B$, or $C$ is tangential to $B$ at $p$, if $C_p(B) =C$ and $C$ is an asymptotic cone of $B$, or $B$ is asymptotic to $C$, if $C$ is an asymptotic cone of $B$.
\end{dfn}
\begin{rmk}
    In this case, the cones are unique i.e. do not depend on any choice of subsequences (see \cite{DS2} or \cite{Song}).
\end{rmk}
A \textit{metric bubbling tree} at $x_0 \in X_0$, denoted by $\mathcal{MBT}_{x_0}$, is defined  as follows: as a set, $\mathcal{MBT}_{x_0}$ consists of equivalence classes of pairs of sections $\sigma : \Delta \to \mathcal{X}$ through $x_0$ and scaling factors $c(t) : \Delta \to \R_{> 0}$ which correspond to isomorphism classes of non-cone bubbling limits at $x_0$. Precisely, 
\begin{equation}
    \mathcal{MBT}_{x_0} = \{(\sigma,c) \mid \lim (X_t, c^2(t)g_t, \sigma(t))\ \text{is a non-cone bubbling limits} \}/\sim
\end{equation}
where $(\sigma_1, c_1) \sim (\sigma_2,c_2)$ if and only if 
\begin{itemize}
    \item $\limsup c_1(t) c^{-1}_2(t) , \limsup c_2(t) c_1^{-1}(t) <\infty$ and
    \item $\limsup c_1(t) d_{g_t} (\sigma_1(t), \sigma_2(t)), \limsup c_2(t) d_{g_t}(\sigma_1(t), \sigma_2(t)) <\infty$.
\end{itemize}
Further, a preorder of $\mathcal{MBT}_{x_0}$ is defined by 
\begin{equation}
    [(\sigma_1,c_1)]  \geq [(\sigma_2, c_2)] :\Leftrightarrow
    \limsup c_2(t)^{-1} c_1(t)<\infty\  \text{and}\   \limsup c_1(t) d_{g_t} (\sigma_1(t), \sigma_2(t)) <\infty.
\end{equation}
It is straightforward that it is independent of a choice of representatives and an equivalence class of $(\sigma(t), c(t))$ gives a unique bubbling limit (up to scaling of metrics and choice of base points).
Hence an element of $\mathcal{MBT}_{x_0}$ is often identified with a genuine bubbling limit at $x_0$ by denoting $B = [(\sigma(t), c(t))]$ to indicate $\lim (X_t, c^2(t)g_t, \sigma(t)) = B$.
\begin{rmk}
    The above $\mathcal{MBT}_{x_0}$ parametrizes \textit{every pointed limit} of the \textit{sequence} $\{(X_t, g_t)\}$ with any scaling $c_t \to \infty$ and base points $x_t \to x_\infty$ even if we allow to take subsequence.
    See Corollary (\ref{cor:independence of taking subsequence}).
\end{rmk}

\section{Main result}\label{sec:main}
\textbf{Setting}: Consider a proper flat family 
\begin{equation}\label{setting:given family}
    \pi :(\mathcal{X}, \mathcal{L}) \to \Delta
\end{equation}
of polarized K3 surfaces over the unit disc $\Delta \subset \C$. 
Assume that
\begin{itemize}
    \item general fibers, i.e. fibers on $t \neq 0$, $X_t := \pi^{-1}(t)$ are smooth,
    \item the central fiber $X_0 :=\pi^{-1}(0)$ has ADE singularities $\{x_1, \ldots , x_k \}$ and
    \item $\mathrm{deg} (L_t) = 2d$ where $L_t := \mathcal{L}\mid_{X_t}$ and $L_t$ is primitive (i.e. there is no line bundle $L'$ so that $L = L'^{\otimes m}$ with $m >1$).
\end{itemize}
Let $\Phi: \Delta \to \mathcal{F}_{2d}$ be the holomorphic curve corresponding to the family (\ref{setting:given family}).
Note that, by the assumptions, $\Phi(0) \in \mathcal{H}$ for the Heegner divisor $\mathcal{H} \subset \mathcal{F}_{2d}$.
To lift the curve $\Phi(t)$ to $L_{2d}\otimes \C$, fix a minimal simultaneous resolution
\begin{equation}
\begin{tikzcd}
    (\widetilde{\mathcal{X}}, \widetilde{\mathcal{L}}) \arrow[r, "\rho"] \arrow[d, "\tilde{\pi}"] & (\mathcal{X}, \mathcal{L}) \arrow[d,"\pi"]\\
    \Delta \arrow[r]& \Delta 
\end{tikzcd},
\end{equation}
where $\widetilde{\mathcal{L}} = \rho^*\mathcal{L}$ and $\Delta \ni z \mapsto z^d \in \Delta$.
Then a simultaneous marking 
\begin{equation}
    \alpha : R^2\tilde{\pi}_* \Z_{\widetilde{\mathcal{X}}}  \to L \times \Delta,
\end{equation}
where $\Z_{\widetilde{\mathcal{X}}}$ is the constant sheaf on $\widetilde{\mathcal{X}}$, can be obtained by fixing isomorphisms
\begin{equation}\label{eq:smooth trivialization of the family}
    \widetilde{\mathcal{X}} \cong_{\mathrm{diff.}} X \times \Delta, H^2(X,\Z) \cong L.
\end{equation}
Further, the curve can be lifted to $L_{2d} \otimes \C$ by fixing a holomorphic 2-form $\Omega_{\widetilde{\mathcal{X}}}$ on $\widetilde{\mathcal{X}}$ and a hermitian metric $h_{\widetilde{\mathcal{L}}}$ on $\widetilde{\mathcal{L}}$ so that 
\begin{itemize}
    \item $\omega_t := \mathrm{c}_1(h_{\widetilde{\mathcal{L}}}\mid _{X_t})$ is the Ricci flat K\"ahler form (in the sense of orbifold for $t=0$),
    \item $\Omega_t := \Omega_{\widetilde{\mathcal{X}}}\mid _{X_t}$ is a nowhere vanishing holomorphic $2$-form and $\frac{1}{2}\Omega_t\wedge \overline{\Omega}_t = \omega_t^2$ and
    \item a period mapping $P: \Delta \to \widetilde\Omega_{2d} \subset L_{2d}\otimes \C$
\begin{equation}
    \begin{aligned}
      P:  \Delta & \to  \widetilde\Omega_{2d} \\
        t & \mapsto \alpha([\Omega_t]),
    \end{aligned}
\end{equation}
is a lifting of the holomorphic curve $\Phi$:
\end{itemize}
\begin{equation}
\begin{tikzcd}
    \Delta \arrow[r, "P"] \arrow[d] & \widetilde\Omega_{2d} \arrow[d]\\
    \Delta \arrow[r,"\Phi"]& \mathcal{F}_{2d} .
\end{tikzcd}
\end{equation}
Again note that $P(0) \in H_{\theta}$ for some root $\theta \in L_{2d}$ by assumption and 
\begin{equation}
    P(0) \in H_{\theta} \Leftrightarrow \Phi(0) \in \mathcal{H}_\theta 
\end{equation}
where $\mathcal{H}_\theta$ is the image of $H_\theta$ under the projection $\Omega_{2d} \to \mathcal{F}_{2d}$.

\subsection{Geometry of the bubbling limits}
Some differential geometric aspects of bubbling limits under the above setting are investigated in this subsection.
Let $B = (B, g_B)$ be a bubbling limit of $\{(X_t,g_t)\}$.
Assume $(\sigma(t), c(t))$ gives the bubbling limit. 
Explicitly, assume the following:
\begin{itemize}
    \item A pointed Gromov-Hausdorff convergence 
    \begin{equation}\label{eq:GH convergence of bubbling}
        (X_t, c^2(t)g_t, \sigma(t)) \to (B,g_b,b)
    \end{equation}
    for some $b \in B$ and 
    \item for any compact subset $K \subset B^{\mathrm{reg}}$, there exists open embeddings 
    \begin{equation}\label{eq:embedding to give CG convergence}
        \iota_t : U \to X_t
    \end{equation}
    on an open neighbourhood $U \subset B^{\mathrm{reg}}$ of $K$ such that 
    \begin{equation}\label{eq:CG convergence}
        \iota_t^* c^2(t)g_t \to g_B, \iota_t^* c^2(t) \omega_t \to \omega_B
    \end{equation}
    in $C^\infty_{K}$ topology.
\end{itemize}

\begin{lem}\label{lem:limit is affine ALE}
    $B$ is an affine ALE gravitational instanton.
\end{lem}
\begin{proof}
    By \cite{Nakajima}, $B$ satisfies the following estimates:
    \begin{equation}\label{eq:maximally volume growth}
        \mathrm{Vol}_{g_b}(\mathrm{B}(b;r)) \geq C r^4
    \end{equation}
    for any $r>0$, with some $C >0$ and $b \in B$, and 
    \begin{equation}\label{eq:finiteness of energy}
        \int_{B} \| \mathrm{Rm}(g_b) \|^2 \mathrm{Vol}_{g_b} < \infty
    \end{equation}
    where $\mathrm{Rm}$ is the total Riemannian curvature tensor.
    The above two estimates (\ref{eq:maximally volume growth}) and (\ref{eq:finiteness of energy}) implies that $B$ is ALE (not necessarily hyperk\"ahler yet) orbifold by \cite{BKN}.
    However, by \cite{Bando} Proposition 5 and 6, $B$ must be hyperk\"ahler.
    Then by Theorem \ref{thm:orbifold surjectivity of torelli of ALE}, $B$ is an affine ALE gravitational instanton.
\end{proof}
\begin{lem}\label{lem:convergent subsequence so that we have hol 2 form}
    For a convergent sequence $(X_t, c_t^2g_t) \to B$, there exists a holomorphic $2$-form $\Omega_B$ on $B$ such that 
    \begin{itemize}
        \item $\omega_B^2 = \frac{1}{2} \Omega_B \wedge \overline{\Omega}_B$ and 
        \item $\iota_{t_j}^*\Omega_{t_j} \to \Omega_B$ in locally smoothly on the regular locus for a suitable subsequence $\{t_j\}_{j=1}^\infty \subset \Delta$ with $t_j \to 0$.
    \end{itemize}
\end{lem}
\begin{proof}
    It seems to be a well-known kind of assertion so we only give a sketch of proof (for a detailed proof, consult the proof of Theorem 5.1.\cite{NO} for example, a similar assertion is proved for the case $B$ is a tangent cone but essentially same as our setting).
    Fix a point $x \in B^{\mathrm{reg}}$, a compact subset $K \subset B^{\mathrm{reg}}$ containing $x$ and an open neighborhood $U \subset B^{\mathrm{reg}}$ so that open embeddings $\iota_t :U \to X_t$ which gives the convergence is defined.
    Then at $x$, $\{\iota_t^* \Omega_{t,x} \}_{t\in \Delta} \subset \wedge ^2T_x^*B$ is a bounded set (as we normalized $\Omega_t$ by $\omega_t^2 = \frac{1}{2}\Omega_t \wedge \overline{\Omega}_t $) hence there exists a convergent subsequence $\{\Omega_{t_j,x}\}_{j=1}^\infty$ which converges to a $2$-covector $\Omega_{B,x}$.
    Then by the parallel transportation with respect to $g_B$ gives a holomorphic $2$ form on $B^{\mathrm{reg}}$ with $\omega_B^2 = \frac{1}{2} \Omega_B \wedge \overline{\Omega}_B$.
\end{proof}
Recall that $E = \bigcup E_j \subset X_0$ is the exceptional divisors.
Then they can be regarded as $E_j \subset X_t$ via the trivialization (\ref{eq:smooth trivialization of the family}) (hence in the fixed underlying differentiable $4$-manifold $X$ of $X_t$).
Notice that $\theta_j = \alpha([E_j])$ spans the sub lattice $\h \subset L_{2d}$.
\begin{prop}\label{prop:diameter estimate of vanishing cycles}
    Let $\zeta(t) : \Delta \to \h_\C$ be a localization of a period map $\Phi : \Delta \to L_{2d}\otimes \C$ of a family $(\mathcal{X},\mathcal{L}) \to \Delta$.
    Let $S \subset X$ be a $2$-dimensional $C^\infty$ submanifold which is contracted to $x_0$, i.e. $d_H (x_0, S) \to 0$ as $t \to 0$ (such $S$ is called as a contracted $2$-cycle in the rest of the paper).
    Assume $S$ represents a root, say $\theta = \sum n_j \theta_j \in \h \subset L_{2d} \cong H^2(X;\Z)$.
    If $\langle \theta, \zeta(t) \rangle = a(t)$, then $\mathrm{diam}_{g_t} (S) =O(|a(t) |^{\frac{1}{2}})$.
\end{prop}
\begin{proof}
    For $S$, consider a scaling $c(t) = (\mathrm{diam}_{g_t}(S))^{-1}$ and a section $\sigma(t)$ such that 
    \begin{equation}
    \sigma(t) \in S \subset X_t.
    \end{equation}
    Then by taking a suitable subsequence, a limit $B=(B,g_B, b)$ exists:
    \begin{equation}\label{eq:Proposition 6, convergence}
    (X_{t_j}, c^2(t_j)g_{t_j}, \sigma(t_j ) ) \to (B,g_B,b)
    \end{equation}
    with holomorphic $2$ forms $\Omega_{t_j} \to \Omega_B$ as in Lemma \ref{lem:convergent subsequence so that we have hol 2 form}.
    If it is either cone or flat, then these are contradictions by the following arguments:
    If $B$ is flat, i.e. isometric to the Euclidean space, let $\iota_{t_j}$ be an open embedding 
    \begin{equation}
    \iota_{t_j} :\mathrm{B}(b, R) \to X_{t_j}
    \end{equation}
    for any $R>0$ which gives the convergence (\ref{eq:Proposition 6, convergence}).
    If $R>0$ is so large that the image contains $S$ (it is possible as $\mathrm{diam}_{g_t} (S) = O(c^{-1}(t))$, say take $R > \limsup \mathrm{diam}_{g_{t_j}}(S)c(t_j)$), then the inverse image $\iota_t^{-1}(S)$ is a non-trivial cycle, which contradicts to $B$ being Euclidean.
    If $B$ is a cone, then it is biholomorphic to $\C^2/\Gamma$ for some $\Gamma \subset \mathrm{SL}(2;\C)$ as it is an asymptotic cone of an affine ALE gravitational instanton.
    Again this is a contradiction since $S$ defines a nontrivial cycle on $B$. Hence it is not flat, in particular it is an affine ALE gravitational instanton.
    In particular, $S$ is not contracted and having a finite diameter in $B$.

    Verify the assertion by an induction on orders of $a(t)$.
    Let $S$ be a contracted $2$-cycle such that $B$ is non-singular.
    Take $R >0$ sufficiently large so that $B\setminus \mathrm{B}(b,R)$ admits a coordinate at infinity.
    Then let $\iota_{t_j} : \mathrm{B}(b,R) \to X_{t_j}$ be the embedding which gives the convergence (\ref{eq:Proposition 6, convergence}) so that the image $\iota_{t_j}(\mathrm{B}(b,R))$ contains $S$.
    Therefore the following holds:
    \begin{equation}
    \langle c^2(t_j) \zeta(t_j), \theta \rangle = c^2(t_j) a(t_j)  \to \langle [\Omega_B], [\iota^{-1}_{t_j}(S)] \rangle \neq 0.
     \end{equation}
     As $a(t)$ is holomorphic with respect to $t$ hence its asymptotic is independent of the choice of subsequences, then the asymptotic
     \begin{equation}
     \mathrm{diam}_{g_t} (S) = c^{-1}(t) = O(|a(t)|^{\frac{1}{2}})
     \end{equation}
     holds.
     For a general $S$, let $\widetilde{B}$ be the minimal resolution of $B$ and take sufficiently large $R>0$ again so that $B\setminus \mathrm{B}(b,R)$ admits a coordinate at infinity.
     Then by the assumption of the induction, for a contracted $2$-cycle $S'\subset \mathrm{B}_{g_{t_j}}(\sigma(t_j), c^{-1}(t)R)$ with $c(t_j) \mathrm{diam}_{g_{t_j}}(S') \to 0$ which represents a root $\theta'$, the following holds:
     \begin{equation}
         \langle c^2(t_j) \zeta(t_j), \theta' \rangle \to 0 .
     \end{equation}
     Therefore $\{c^2 (t_j)\zeta(t_j)\}$ converges to a vector in $IH^2(B)$.
     On the other hand, the convergence $c^2(t_j) g_{t_j} \to g_B, c(t_j)^2\Omega_{t_j} \to \Omega_B$ implies that $c^2(t_j) \zeta(t_j) \to [\Omega_B]$ in $IH^2(B)$.
    In particular, the following asymptotic holds:
    \begin{equation}
    \langle c^2(t_j) \zeta(t_j), \theta \rangle = c^2(t_j) a(t_j) = O(1)
    \end{equation}
    which completes the proof.
\end{proof}
\subsection{Proof of the main theorem}
Recall that there are two bubbling trees $\mathcal{PBT}_{\zeta}$, a period bubbling tree, and $\mathcal{MBT}_{x_0}$, a metric bubbling tree, for a given family $(\mathcal{X}, \mathcal{L}) \to \Delta$ (see section \ref{section:bubbling trees}).
Our main result is the following.
\begin{thm}\label{thm:main theorem}
   Let 
   \begin{equation}
       (\mathcal{X}, \mathcal{L}) \to \Delta 
   \end{equation}
   be a flat proper family of K3 surfaces with smooth general fibers $X_t$ and the central fiber $X_0$ admitting ADE singularities and let $\Phi : \Delta \to \Omega_{2d}$ be a period mapping of the family.
   For a singularity $x_0 \in X_0$, let $\zeta : \Delta \to \h_\C$ be the localization of $\Phi$ along $\h_\C$ (see section \ref{sec:localization of period}).
   Then there is a poset isomorphism 
   \begin{equation}
      f:\mathcal{PBT}_{\zeta} \to \mathcal{MBT}_{x_0}
   \end{equation}
  such that if $f(([\zeta_v], \h_v)) = B$ then $B \cong Y_{\zeta_v}$ as affine ALE instantons, where $Y_{\zeta_v}$ is Kronheimer's affine ALE gravitational instanton corresponding to $\zeta_v \in \h_v $ (cf. theorem \ref{thm:Hyperkahler structure on the minimal resolution of Kleinnian singularity} and theorem \ref{thm:Hyperkahler structure on deformations on Kleinian singularity}).
\end{thm}
\begin{proof}
First, fix data as in the setting (see the beginning part of this section for details).
In particular, for a given family 
\begin{equation}\label{proof of the main theorem: given familt}
    (\mathcal{X}, \mathcal{L}) \to \Delta,
\end{equation}
fix a simultaneous resolution 
\begin{equation}\label{proof of the main theorem: sim resol}
    (\widetilde{\mathcal{X}}, \widetilde{L}) \to \Delta,
\end{equation}
trivialization
\begin{equation}\label{proof of the main theorem: trivialization}
    \widetilde{\mathcal{X}} \cong X \times \Delta,
\end{equation}
marking and embedding 
\begin{equation}\label{proof of the main theorem: markings}
    H^2(\widetilde{\C^2/\Gamma} ; \Z)  \cong \h \hookrightarrow L_{2d} \subset L \cong H^2(X;\Z)
\end{equation}
and hyperk\"ahler structures
\begin{equation}\label{proof of the main theorem: forms}
    \omega_t \in c_1(L_t), \Omega_t\ \text{with}\ \omega_t^2 = \frac{1}{2}\Omega_t \wedge \overline{\Omega}_t. 
\end{equation}
Let $\Phi: \Delta \to L_{2d} \otimes \C$ be the period map and $\zeta : \Delta \to \h_\C$ be its localization.
Then $f$ is defined by the following inductive way: For the root $([\zeta_0], \h_\C) \in \mathcal{PBT}_{\zeta}$, the following expression of $\zeta(t)$ holds by definition:
\begin{equation}
    \zeta (t) = t^{k_0} \zeta_0 + O(t^{k_0 +1}),
\end{equation}
for some $k_0 \geq1$.
Then $f([\zeta_0], \h_\C)$ is defined as 
\begin{equation}
    f([\zeta_0], \h_\C) = [(\sigma_0(t), c_0(t))],
\end{equation}
where $\sigma_0(t)$ is any section and $c_0(t) = |t|^{-\frac{k_0}{2}}$. 
To see this is well-defined, i.e. the sequence $\{(X_t, c_0(t)^2 g_t, \sigma_0(t))\}$ is convergent as $t \to 0$ and defines a non-cone bubbling limit, take a convergent subsequence $\{(X_{t_j}, c_0(t_j)^2g_{t_j}, \sigma_0(t_j))\}$ with $t_j \to 0$ by the precompactness theorem by Donaldson-Sun \cite{DS2}.
Let $B_{\zeta_0} = (B_{\zeta_0}, g_{\zeta_0}, b_{\zeta_0})$ be the limit.
Furthermore, take a nowhere vanishing holomorphic $2$-form $\Omega_{\zeta_0}$ with 
\begin{equation}\label{proof of main theorem: normalization of hol 2 form}
    \omega_{\zeta_0}^2 = \frac{1}{2} \Omega_{\zeta_0} \wedge \overline{\Omega}_{\zeta_0}.
\end{equation}
By Lemma \ref{lem:convergent subsequence so that we have hol 2 form}, $\Omega_{\zeta_0}$ can be taken so that
\begin{equation}\label{proof of main theorem: convergence of hol 2 forms for root}
    \Omega_{t_j} \to \Omega_{\zeta_0}
\end{equation}
with taking further subsequence, again we denote it by $\{t_j\}$.
Note that the choice of $\Omega_{\zeta_0}$ is unique up to $\mathrm{U}(1)$ multiplication.
Recall that $B_{\zeta_0}$ is an affine ALE gravitational instanton (see Lemma \ref{lem:limit is affine ALE}).
Take $R>0$ sufficiently large so that 
\begin{itemize}
    \item $\mathrm{B}_{g_{t_j}}(\sigma_0(t_j), c_0^{-1}(t_j)R) \to \mathrm{B}_{g_{\zeta_0}}(b_{\zeta_0}, R)$ in the Gromov-Hausdorff sense and
    \item $B_{\zeta_0} \setminus \mathrm{B}(b_{\zeta_0}, R)$ admits a coordinate at infinity.
\end{itemize}
Then by Proposition \ref{prop:diameter estimate of vanishing cycles}, all $E_j$ are contained in $\mathrm{B}_{g_{t_j}}(\sigma_0(t_j), c_0^{-1}(t_j)R)$ for any sufficiently small $t_j$.
Furthermore, a root $\theta$ is perpendicular to $\zeta_0$ if and only if its representative $S$ is contracted in $B_{\zeta_0}$ by Proposition \ref{prop:diameter estimate of vanishing cycles} again.
Therefore the convergences 
\begin{equation}
c_0^2(t_j) g_{t_j} \to g_{\zeta_{0}}, c_0^2(t_j) \Omega_{t_j} \to \Omega_{\zeta_0}
\end{equation}
 implies that 
 \begin{equation}
 \zeta_0 = \lim_{j \to \infty} c_0^2(t_j) \zeta(t_j) = [\Omega_{\zeta_0}] \text{ in } IH^2(B).
 \end{equation}
The Torelli theorem for affine ALE instantons (Theorem \ref{thm:orbifold injectivity of period of ALE}) implies that $B_{\zeta_0} \cong Y_{\zeta_0}$ where $Y_{\zeta_0}$ is Kronheimer's ALE instanton (Theorem \ref{thm:Hyperkahler structure on deformations on Kleinian singularity}).
Then the limit is independent of the choice of subsequences and defines a non-cone limits.
Note that the asymptotic cone $\mathcal{C}_\infty(B_{\zeta_0})$ is isomorphic to $0 \in \C^2/\Gamma   \cong x_0 \in X_0 $ as analytic germs, hence $B_{\zeta_0}$ is the minimal bubbling.

For a general vertex $([\zeta_j], \h_j) \in \mathcal{PBT}_{\zeta}$, assume that $f$ is well-defined for the parent $([\zeta_{j-1}], \h_{j-1})$ of $([\zeta_{j}], \h_{j})$ and put $f([\zeta_{j-1}], \h_{j-1}) =[(\sigma_{j-1}, c_{j-1})] = B_{\zeta_{j-1}} \cong Y_{\zeta_{j-1}}$.
By definition of the order $([\zeta_{j-1}], \h_{j-1}) >( [\zeta_j],\h_{j})$, $\zeta(t)$ satisfies the following expression:
\begin{equation}
    \pi_{\h_{\zeta_j}} \circ \zeta (t) = t^{k_j} \zeta_{j} + O(t^{k_j +1})
\end{equation}
where $\h_{\zeta_j}$ is a maximal irreducible sub-root system perpendicular to $\zeta_{j-1} \in \h_{\zeta_{j-1}}$.
Furthermore $B_{\zeta_{j-1}}$ has singularities one to one corresponding to maximal irreducible sub-root systems perpendicular to $\zeta_{j-1}$, say $x_j \in B_{\zeta_{j-1}}$ corresponds to $\h_{\zeta_j}$.
Then $f([\zeta_{j}],\h_j) = [(\sigma_j, c_j)]$ is defined by 
\begin{equation}
    c_j (t) = |t|^{-\frac{k_{j}}{2}} c_{j-1}(t)
\end{equation}
and 
\begin{equation}
    \sigma_j(t) \to x_j \text{ in } (X_t, c_{j-1}^2(t) g_t, \sigma_{j-1}(t)) \to (B_{j-1}, g_{j-1}, b_{j-1}).
\end{equation}
Then take a convergent subsequence $(X_{t_k}, c_j(t_k)^2 g_{t_k}, \sigma_j({t_k})) \to (B_{\zeta_j}, g_{\zeta_j}, b_{\zeta_j})$ and holomorphic $2$ form $\Omega_{\zeta_j}$ again as in the case of the root (\ref{proof of main theorem: normalization of hol 2 form}) (\ref{proof of main theorem: convergence of hol 2 forms for root}).
Note that Proposition \ref{prop:diameter estimate of vanishing cycles} implies that a $2$-cycles $S$ representing a root $\theta$ satisfies
\begin{equation}
    c_{j}(t) d_H(x_j, S) < \infty,
\end{equation}
as $(X_{t_k}, c_j^2(t_k) g_{t_k}, \sigma_j(t_k)) \to (B_{\zeta_j}, g_{\zeta_j}, b_{\zeta_j})$ if and only if $\theta \in \h_{\zeta_j}$.
The similar argument as the proof of Proposition \ref{prop:diameter estimate of vanishing cycles}, or the argument for the root, implies that $\zeta_j \in IH^2(B_{\zeta_j})$ and $B_{\zeta_j}\cong Y_{\zeta_j}$.
Therefore $f: \mathcal{PBT}_{\zeta} \to \mathcal{MBT}_{x_0}$ is a well-defined map.\\
\textbf{Injectivity}: By its construction, $f$ preserves the orders strictly.
For $([\zeta_1], \h_1), ([\zeta_2], \h_2) \in \mathcal{PBT}_{\zeta}$ which are not comparable, take the smallest common ancestor $([\zeta_3], \h_3)$, i.e. the smallest one such that $([\zeta_3], \h_3) > ([\zeta_1], \h_1), ([\zeta_2], \h_2)$.
Then there are two children $([\zeta_1'], \h'_1), ([\zeta_2'], \h_2') < ([\zeta_3], \h_3)$ such that $([\zeta_1], \h_1) < ([\zeta_1'], \h_1')$ and $([\zeta_2], \h_2) < ([\zeta_2'], \h_2')$.
Notice that by the definition, the root systems $\h_{\zeta_1'}, \h_{\zeta_2'}$ are disjoint.
Then Proposition \ref{prop:diameter estimate of vanishing cycles} implies 
\begin{equation}
    c_{\zeta_1'}(t) d(\sigma_{\zeta_1'}(t), \sigma_{\zeta_2'}(t)) \to \infty.
\end{equation}
Hence $f([\zeta_1'], \h_1') \neq f([\zeta_2'],\h_2')$.
If $f([\zeta_1], \h_1) = f([\zeta_2], \h_2)$, then it has ancestors $f([\zeta_1'], \h_1')$ and $f([\zeta_2'], \h_2')$ which are not comparable and it contradicts to $\mathcal{MBT}_{x_0}$ being a tree.
Then $f$ is an injective poset map.\\
\textbf{Surjectivity}: If $f$ is not surjective, let $B$ be a maximal element which is not contained in the image.
Notice that $B= [(\sigma,c)]$ is not the minimal bubble since the root of $\mathcal{PBT}_{\zeta}$ is mapped to the minimal bubble.
In particular, there exists the parent $B' = [(\sigma',c')]$.
By the assumption, there exists an inverse image $([\zeta'], \h')$.
Let $x \in B' \cong 0 \in \C^2/\Gamma_{B} $ be the limit of $\sigma(t)$ in $(X_t, c'^2(t)g_t, \sigma'(t)) \to B'$.
Then it corresponds to a maximal irreducible sub root system $\h_{B}$ perpendicular to $\zeta'$.
For each singularity $x \in B'$ there is exactly only one child of $B'$.
    In fact, if there are two distinct children $B_1 =[(\sigma_1,c_1)], B_2 =[(\sigma_2, c_2)]$ corresponding to the same $x \in B'$, then the following holds: 
    \begin{equation}
        c_1(t)c_2(t) =O(1), c_1(t) d_{g_t}( \sigma_1(t), \sigma_2(t)) \to \infty.
    \end{equation}
    Then there must be a $2$-cycle $S$ representing a root $\theta$ with 
    \begin{equation}
        c(t) \mathrm{diam}_{g_t}(S) \to 0, c_1(t) \mathrm{diam}_{g_t}(S)  \to \infty. 
    \end{equation}
    By Proposition \ref{prop:diameter estimate of vanishing cycles}, $\theta$ satisfies
    \begin{equation}
        \langle \theta, \zeta'\rangle=0
    \end{equation}
    and the maximal irreducible sub-root system $\h_\theta$ perpendicular to $\zeta'$ which contains $\theta$ must define a bubbling limit $[(\sigma_\theta, c_\theta)]$ with 
    \begin{equation}
        c_\theta(t) c_1(t) \to \infty.
    \end{equation}
    This contradicts to $B_1, B_2$ being children of $B'$.
As $B$ and $B_\zeta = f([\zeta], \h_B)$ are corresponding to the same singularity, they must coincide, it contradicts to the assumption.
\end{proof}
The above proof directly implies the following corollary.
\begin{cor}\label{cor:independence of taking subsequence}
Let $c_t \to \infty$ and $x_t \in X_t$ be \textit{any} scaling factor and base points with $x_t \to x_0$ (i.e. not necessarily assumed to come from a section). 
Then for any convergent subsequence $\{(X_j, c_j^2 g_j, x_j)\}$ of $\{(X_t, c_t^2 g_t, x_t)\}$ there exists $ B = [(\sigma(t), c(t))] \in \mathcal{MBT}_{x_0}$ such that 
\begin{equation}
    B \cong \lim (X_j, c_j^2g_j, x_j).
\end{equation}
\end{cor}

\section{Examples}\label{sec:Examples}
This section is devoted to give some examples and comparisons with previous studies.
\subsection{Explicit examples for $A_{k}$ type singularities}
A cyclic quotient $\C^2/\Z_{k+1}$ is called the $A_k$-type singularity.
Their deformation space has an explicit expression as follows:
Consider the Lie algebra of special linear group $\mathfrak{sl}(k+1)$. 
Then its subspace $\h_\C$ consisting of diagonal matrices is a Cartan sub-algebra of $\mathfrak{sl}(k+1)$.
Let $e_j$ be the matrix unit $(\delta_{jj})$.
Then $\{e_j - e_{j+1} \}_{j=1}^{k}$ forms a basis of $\h_\C$.
If $\h_\C$ is equipped with an inner product $\langle\cdot , \cdot \rangle$ which makes $\{e_j \}$ orthonormal, then vectors $ \{e_j - e_{l}\}_{1\leq j,l \leq k+1}$ forms the set of roots with respect to $-\langle \cdot, \cdot \rangle$.
Then the deformation space of an $A_k$ singularity is given by $\h_\C /W$:
\begin{equation}
    \mathrm{Kur}( \C^2/\Z_{k+1} ) = \h_\C /W,
\end{equation}
where $W$ is the Weyl group of $\h_\C$, with discriminant loci
\begin{equation}
    H_\theta = \{\langle x, \theta \rangle =0 \} \subset \h_\C,
\end{equation}
where $\theta$ is a root.
For a given deformation $\mathcal{X} \to \Delta$ of the $A_k$-singularity corresponding to a curve $\Phi : \Delta \to \h_\C /W$, a lifting of the curve $\Phi$, say $\widetilde{\Phi}$, to the Weyl covering $\h_\C$ gives a base change of the family and it admits a minimal simultaneous resolution.
\begin{eg}
Consider an $A_3$ singularity.
Then its Weyl covering of the deformation space $\h_\C$ is a $3$-dimensional vector space.
Take simple roots by $\{\theta_1, \theta_2, \theta_3\}$.
Then their positive roots are 
\begin{equation}
    \begin{aligned}
        \theta_1, &\theta_2, \theta_3, \\
        \theta_1 + \theta_2, &\theta_2 + \theta_3, \theta_1 + \theta_2 + \theta_3 .
    \end{aligned}
\end{equation}
Consider a family defined by 
\begin{equation}
    \zeta(t) = \left(t^2 + \frac{1}{2} t \right) \theta_1 + (t^2 + t) \theta_2 + \left( t^2 + \frac{1}{2}t \right) \theta_3.
\end{equation}
A simple calculation gives 
\begin{equation}
   \left\{ \begin{aligned}
        \langle \theta_1 , \zeta(t) \rangle &= - t^2 \\
        \langle \theta_2, \zeta(t) \rangle  &= -t \\
        \langle \theta_3, \zeta(t) \rangle & = -t^2 .
    \end{aligned}
    \right.
\end{equation}
Then $\zeta(t)$ defines a family with smooth fibres for $0 < |t| << 1$.
Furthermore, a direct calculation implies
\begin{equation}
    \zeta_0 = \zeta(t)/t \mid_{t=0} = \frac{1}{2} \theta_1 + \theta_2 + \frac{1}{2} \theta_3.
\end{equation}
This implies that the minimal bubble has two $A_1$ singularities, corresponding to $\theta_1$ and $\theta_3$.

\end{eg}
\subsection{Comparison with local models}
\subsubsection{Local models by de Borbon-Spotti}
There is another description of the deformation space and its Weyl covering of the $A_k$ singularity in terms equations defining a deformation family.
The $A_k$ singularity has the following standard form:
\begin{equation}
    X_0 = [xy = z^{k+1}] \subset \C^3 .
\end{equation}
Consider the following variety
\begin{equation}\label{eq:versal family of Ak singularity}
    \mathcal{X} = \left[xy = z^{k+1} + \alpha_1  z^k + \cdots + \alpha_{k+1} \right] \subset \C^3_{(x,y,z)} \times \C^{k+1}_{\alpha}.
\end{equation}
The variety $\mathcal{X}$ is a flat family with respect to the second projection $\pi : \mathcal{X} \to \C^{k+1}$.
Then $\mathcal{X} \to \C^{k+1}$ is the universal family of the deformation space of the $A_k$ singularity.
Hence, in particular, for a deformation of the $A_k$ singularity
\begin{equation}
    \mathcal{X} \to \Delta 
\end{equation}
over the unit disc $\Delta \subset \C$, there is a $(k+1)$-tuple $(\alpha_j(t))$ of holomorphic functions over $\Delta$ such that 
\begin{equation}
    \mathcal{X} \cong \left[ xy = z^{k+1} + \alpha_1(t) z^{k} + \cdots + \alpha_{k+1}(t)  \right] \subset \C^3_{(x,y,z)} \times \Delta_t .
\end{equation}
Its Weyl covering is given by 
\begin{equation}
    \begin{aligned}
        \C^{k+1} & \to \C^{k+1}\\
        (a_0, \ldots, a_k) &\mapsto (\alpha_j(a_0,\ldots,a_k))_j
    \end{aligned}
\end{equation}
where $\alpha_j$ is the fundamental symmetric polynomial of order $j$.
Furthermore, the base change of the family is given by 
\begin{equation}
    [xy = \Pi_j (z - a_j(t))] \subset \C^3 \times \Delta.
\end{equation}
Each fibre $X_t = [xy = \Pi(z -a_j(t))]$ can be equipped with a unique $dd^c$-exact ALE metric $g_t$ by \textit{Gibbons-Hawking ansatz}.
Under the above situation, bubbling trees can be described in terms of behavior of $(a_j(t))$ \cite{dBS}.
To see this, a tree $\mathcal{T}$ is constructed from $(a_j(t))$ in the following way:
An equivalence relation $\sim_n$ of holomorphic functions over $\Delta$ for each $n \in \mathbb{N}$ defined as follows:
\begin{equation}
    f \sim_n g :\Leftrightarrow \mathrm{ord}_{t=0} (f- g) \geq n.
\end{equation}
A tree $\mathcal{T}_0$ is defined as follows:
\begin{itemize}
    \item the root is $\{a_0, \ldots, a_{k}\}$ with its level $0$ and 
    \item children of a vertex $\{a_{l_1}, \ldots, a_{l_m} \}$ with its level $n$ consists of the equivalence classes of $\{a_{l_1}, \ldots, a_{l_m} \}$ by $\sim_{n+1}$ with their level $n+1$.
\end{itemize}
Then the tree $\mathcal{T}$ is obtained by contracting \textit{one child vertices}.
\begin{thm}[Theorem 3. \cite{dBS}]\label{thm:dbs}
    For a family $\mathcal{X} = [xy = \Pi (z -a_j(t) ) ]$, the bubbling tree of $\{(X_t, g_t)\}$ is given by $\mathcal{T}$.
\end{thm}
The two descriptions of deformation spaces are translated as follows and then Theorem \ref{thm:main theorem} and Theorem \ref{thm:dbs} are compatible.
The holomorphic functions $(a_j)$ can be assumed to satisfy
\begin{equation}
    \sum a_j =0
\end{equation}
in (\ref{eq:versal family of Ak singularity}) without loss of generality.
Then there is a bijection 
\begin{equation}\label{eq:transelation of two descriptions}
    \begin{aligned}
        \h_\C & \to \C^{k+1} \\
        \sum a_j \theta _j & \mapsto \sum a_j(e_j - e_{j+1})
    \end{aligned}
\end{equation}
which induces an isomorphism of versal families of two deformations.
Therefore it is sufficient to see that $\mathcal{PBT}_{\zeta}$ and $\mathcal{T}$ gives the same tree under the correspondence (\ref{eq:transelation of two descriptions}).
Let $n$ be the smallest number so that $\{a_j(t) \}$ has nontrivial equivalence classes with respect to $\sim_n$.
Then $\zeta(t)$ is divisible at most $t^{n-1}$.
In particular $\theta = e_j -e_l$ satisfies
\begin{equation}
\begin{aligned}
    \langle \theta , \zeta_0 \rangle &= \langle \theta, \zeta(t)/t^{n-1} \rangle \mid_{t=0} \\
    &= (a_j(t) - a_l(t))/t^{n-1} \mid _{t=0}.
    \end{aligned}
\end{equation}
This implies that 
\begin{equation}
    \theta \perp \zeta_0 \Leftrightarrow a_j \sim_{n+1} a_l.
\end{equation}
Therefore, a maximal irreducible sub root system perpendicular to $\zeta_0$ is given by $\{a_j - a_l \} $ for a child of $\{a_j\}$.
 
\subsubsection{Algebraic theory by Odaka}
In \cite{odaka25}, a candidate of algebraic construction of bubbling limits is proposed (which can be applied for more general situations than our ADE cases).
The following is a review of the construction for the case of surfaces with ADE singularities (see Section 2.2.3, proof of Theorem 2.4. and Section 2.2.4., proof of Theorem 2.13.).
Consider an ADE-type singularity $0 \in X_0$ and its deformation space $\mathrm{Kur}(X_0)$.
In this case, 
\begin{equation}
    \mathrm{Kur}(X_0) \cong \h_\C / W
\end{equation}
where $\h_\C$ is a Cartan subalgebra corresponds to $X_0$ and $W$ is the Weyl group (if $X_0$ is the $A_n$ singularity, $\h_\C = \mathfrak{sl}(n+1)$ and $W = \mathfrak{S}_n$ acting on $\mathfrak{sl}(n+1)$ via the permutations of coordinates, for example).
For a given deformation family $\mathcal{X} \to \Delta$ of $X_0$ over the unit disc, there exists a holomorphic curve 
\begin{equation}
    \varphi: \Delta \to \mathrm{Kur}(X_0)
\end{equation}
and it lifts to a covering of $\mathrm{Kur}(X_0)$ after a suitable base change of the family $\mathcal{X} \to \Delta$:
\begin{equation}
\begin{tikzcd}
    \Delta \arrow[r, "\varphi'"] \arrow[d] & \h_\C \arrow[d]\\
    \Delta \arrow[r,"\varphi"]& \mathrm{Kur}(X_0) = \h_\C/W .
\end{tikzcd}
\end{equation}
Let 
\begin{equation}
    \varphi'(t) = \zeta_0 t^d + O(t^{d+1})
\end{equation}
with $\zeta_0 \neq 0$ and consider a map 
\begin{equation}
\begin{aligned}
    \varphi_{\mathrm{min}}' :\Delta &\to \h_\C\\
    t &\mapsto \varphi'(t)/t^d = \zeta_0 + O(t).
\end{aligned}
\end{equation}
Then the central fiber of the family $\mathcal{X}'_{\mathrm{min}}$ given by $\varphi'_{\mathrm{min}}$ is the minimal bubble of the algebro-geometric construction.
Deeper bubbles can be obtained by restricting the family $\mathcal{X}'_{\mathrm{min}}$ around a singularity of the central fiber and repeating the same procedure.
This is nothing but the construction of our period bubbling tree and hence the main theorem (Theorem \ref{thm:main theorem}) implies that Odaka's algebraic construction of bubbling limits gives genuine bubbling limits for polarized K3 surfaces.

\begin{rmk}
    The central fiber of the family $\mathcal{X}'_{\mathrm{min}}$ is obtained by a weighted blow-up of $\mathcal{X}$ at the singularity $x_0 \in X_0$ as follows.
    After a suitable base change if it is necessary, embed the family $\mathcal{X}$ to $\C^3 \times \Delta$ with the weighted $\C^*$ action (see \cite{Slo} for example) and blow-up $x_0 \in \mathcal{X}$ with respect to the weight times $d$ if $\varphi' = t^d \zeta_0 + O(t^{d+1})$.

    \begin{eg}
        If $\mathcal{X}$ is $A_n$ type, we have 
        \begin{equation}
            \mathcal{X} = [xy = \prod_{j=0}^n (z - a_j(t)) ] \subset \C^3_{x,y,z} \times \Delta_t,
        \end{equation}
        after a suitable base change with $a_j(0) = 0$.
        The weighted $\C^*$ action is given by 
        \begin{equation}
            \alpha \cdot (x,y,z,t) := (\alpha^{n+1} x, \alpha^{n+1} y, \alpha^2 z, \alpha t).
         \end{equation}
        Assume $a(t) = (a_0(t), a_1(t), \ldots, a_{n}(t) ) = t^{2d} a' + O(t^{2d+1})$ and let 
        \begin{equation}
            \mathcal{X}' \to \mathcal{X} 
        \end{equation}
        be the weighted blow-up with the weight $((n+1)d, (n+1)d, 2d,1)$.
        Then, locally, we have 
        \begin{equation}
            (x', y', z', t) \mapsto (t^{(n+1)d} x', t^{(n+1)d} y', t^{2d} z', t)
        \end{equation}
        for the weighted blow-up.
        In particular, the strict transform of the family is
        \begin{equation}
            [t^{2d(n+1)} x'y' =t^{2d(n+1)} \prod_{j=0}^n (z' - a_j(t)/t^{2d} ) ].
        \end{equation}
        Therefore, the complement of the strict transform of $X_0$ in the central fiber of $\mathcal{X}'$ is 
        \begin{equation}
            [xy= \prod_{j=0}^n (z- a'_j)] 
        \end{equation}
        where $a' =(a'_0, \ldots, a_n')$, the minimal bubble at $x_0$.
    \end{eg}
    By repeating the weighted blow-up at ADE singularities of the central fibers, we obtain a family 
    \begin{equation}
        \mathcal{X}^B \to \Delta
    \end{equation}
    with the central fiber $X_0^B = X'_0 \cup \bigcup_{j=1}^k (\cup_{v\in \mathcal{T}_j} B'_v )$, where $\mathcal{T}_j$ is the bubbling tree at $x_j \in X_0$, such that the dual intersection graph of each $\cup_{v \in \mathcal{T}_j} B_v$ is the graph of $\mathcal{T}_j$ and 
    \begin{equation}
        \begin{aligned}
            X_0' \setminus \text{(intersection loci)} &\cong X_0^{\mathrm{reg}}, \\
            B_v' \setminus{(\text{intersection loci)}} & \cong B_v^{\mathrm{reg}},
        \end{aligned}
    \end{equation}
    where $B_v$ is the affine ALE labeled by $v$.
\end{rmk}

\section{Discussion}
\subsection{Toward a multi-scale K-moduli space}
In this subsection, we construct a complex analytic space $\hat{\mathcal{F}}_{2d}$ with a proper birational map $\hat{\mathcal{F}}_{2d} \to \mathcal{F}_{2d}$ which is biregular on $\mathcal{F}_{2d}^o$.
This is a candidate of the so-called \textit{multi-scale K-moduli space} which parameterizes all non-cone bubbling limits (see also \cite{dBS} Section 4).
To construct $\hat{\mathcal{F}}_{2d}$ some notions are necessary.
\begin{dfn}
    Let $\{\theta_1, \ldots, \theta_n\} \subset L_{2d}$ be a set of roots.
    Assume $\h_\Z = \mathrm{Span}_\Z \{\theta_1, \ldots, \theta_n \} \subset L_{2d}$ is isomorphic to the root lattice of a Cartan sub-algebra of type ADE.
    Then an analytic set $S \subset \Omega_{2d}$ of ADE type is an irreducible component of $\cap_{j=1}^n H_{\theta_j}$.
\end{dfn}
Let $\mathcal{S}$ denote the set of analytic sets of ADE type of $\Omega_{2d}$.
Note that $\Gamma (=\text{stabilizer of}\ \lambda)$ acts on $\mathcal{S}$ and there are only finitely many orbits as an orbit corresponds to a stratum of the Heegner divisor $\mathcal{H} \subset \mathcal{F}_{2d}$.
In particular, $\mathcal{S}$ is locally finite. 
To obtain the $\hat{\mathcal{F}}_{2d}$, consider the following successive blow-ups:
Let $\Omega^{(1)}_{2d}$ be the blow-up along all minimal elements (with respect to the inclusion) in $\mathcal{S}$:
\begin{equation}
   b^{(1)} : \Omega^{(1)}_{2d} \to \Omega_{2d}.
\end{equation}
Note that they are locally finite and $b^{(1)}:\Omega_{2d}^{(1)} \to \Omega_{2d}$ is independent of the choice of the order of blow-ups.
In fact, consider a point $p \in S := S_1 \cap \ldots \cap S_k$, where $S_j$ are minimal analytic sets of  ADE type. 
Consider the normal spaces of $\widetilde{S}_j \subset \widetilde{\Omega}_{2d}$, the preimage of $S_j$.
If $S_j$ is cut out by $\{\theta_{j,1},\ldots, \theta_{j,n_j} \}$ then, for a preimage $\tilde{p} \in \tilde{S} =\tilde{S}_1 \cap \ldots \cap \tilde{S}_k$, we have 
\begin{equation}
    T_{\tilde{p}} \tilde{\Omega}_{2d} = T_{\tilde{p}} \tilde{S} \oplus \h_1 \oplus \cdots \oplus \h_k (\subset L_{2d} \otimes \C),
\end{equation}
where $\h_j = \mathrm{Span}_\C \{\theta_{j,1}, \ldots, \theta_{j,n_j}\}$.
In fact, we have $\h_j \cap \h_l = \{0\}$ for $j \neq l$ since each $S_j$ are minimal among ADE type analytic sets.
Note that $\h_j$ is isomorphic to the normal space of $\tilde{S}_j$ at $\tilde{p}$ which implies that $b^{(1)}$ is independent of the choice of orders of blow-ups.
We define the set of analytic sets of ADE type $\mathcal{S}^{(1)}$ in $\Omega_{2d}^{(1)}$ to be the set of strict transforms of analytic sets of ADE type in $\Omega_{2d}$ which are not minimal.
Define 
\begin{equation}
b^{(k)} : \Omega^{(k)}_{2d} \to \Omega_{2d}^{(k-1)}
\end{equation}
to be the blow-up along all minimal analytic sets of $\Omega^{(k-1)}_{2d}$ inductively.
Note that if $S \subset S'$ for analytic sets of type ADE, then $\h' \subset \h$ for the corresponding Cartan subalgebras.
Hence, in particular, there exists a $k$ so that $\Omega_{2d}^{(k)}$ has no analytic sets of type ADE with codimension greater than $2$.
Then let $\hat{\Omega}_{2d}$ be the $\Omega_{2d}^{(k)}$.
The group action $\Gamma \curvearrowright \Omega_{2d}$ naturally extends to $\hat{\Omega}_{2d}$ and $\hat{\mathcal{F}}_{2d}$ is the quotient by $\Gamma$:
\begin{equation}
    \hat{\mathcal{F}}_{2d} := \Gamma \backslash \hat{\Omega}_{2d}.
\end{equation}
\begin{lem}
    The action $\Gamma \curvearrowright \hat{\Omega}_{2d}$ is properly discontinuous.
 \end{lem}
 \begin{proof}
     Show the following assertion: For all $l \leq k$, the action $\Gamma \curvearrowright \Omega^{(l)}_{2d}$ is properly discontinuous.
     Note that for any subset $A \subset \Omega^{(1)} _{2d}$, we have
     \begin{equation}
         \gamma A \cap A \neq \varnothing \Rightarrow \gamma b^{(1)}(A) \cap b^{(1)}(A)  \neq \varnothing.
     \end{equation}
     Hence, in particular,
     \begin{equation}
         \{ \gamma \in \Gamma \mid \gamma A \cap A \neq \varnothing \} \subset \{\gamma \in \Gamma \mid \gamma b^{(1)} (A) \cap b^{(1)}(A) \neq \varnothing \}.
     \end{equation}
     Then the properly discontinuity of $\Gamma \curvearrowright\Omega_{2d}$ implies the properly discontinuity of $\Gamma \curvearrowright \Omega_{2d}^{(1)}$.
     The similar argument for $\Gamma \curvearrowright \Omega_{2d}^{(l)}, b^{(l)} : \Omega_{2d}^{(l)} \to \Omega_{2d}^{(l-1)}$ shows that the actions are properly discontinuous inductively.
 \end{proof}
\begin{cor}
    The quotient $\hat{\mathcal{F}}_{2d}$ is an analytic space.
\end{cor}
The space $\hat{\mathcal{F}}_{2d}$ parameterizes the set of non-cone bubbling limits in the follwing sense.
\begin{prop}
    Let $p \in \mathcal{F}_{2d}$ be a point in the Heegner divisor (i.e. $p$ corresponds to a K3 surface with ADE singularities).
    Consider the set of germs of holomorphic curves passing through $p$. 
    For two germs $\varphi$ and $\psi$ passing through $p$, take a lifting $\Phi$ and $\Psi$ to $\Omega_{2d}$ and consider period bubbling trees $\mathcal{T}_x$ and $\mathcal{T}_x'$ defined by $\Phi$ and $\Psi$ at a singularity $x \in X_0$ ($=$ the K3 surface corresponds to $p$).
    If there exists an $\alpha \in \Gamma$ such that $\zeta \in \mathbb{P}(\h) \mapsto \alpha \zeta \in \mathbb{P}(\alpha \h)$ induces poset isomorphisms $\mathcal{T}_x \to \mathcal{T}_x'$ for all $x \in X_0^{\mathrm{sing}}$, $\varphi$ and $\psi$ are said to be equivalent.
    Then $\varphi$ is equivalent to $\psi$ if and only if $\hat{\varphi}(0) = \hat{\psi}(0)$ for the liftings
    \begin{equation}
        \hat{\varphi}, \hat{\psi}: \Delta \to \hat{\mathcal{F}}_{2d}.
    \end{equation}
\end{prop}
\begin{proof}
  Let $P \in S = S_1 \cap \cdots \cap S_n \subset \Omega_{2d}$ be a preimage of $p \in \mathcal{F}_{2d}$.
  We fix a coordinate $(z_{1,1}, \ldots, z_{j,l}, \ldots, z_{n,n_n}, \ldots, z_{19})$ centered at $P$ so that 
  \begin{equation}
      S_j = [z_{j,1} = \cdots = z_{j, n_j} =0],
  \end{equation}
  so that $\{z_{j,l} \}_{l=1}^{n_j} \cap \{z_{k,l}\}_{l=1}^{n_k} = \varnothing$ for $ j\neq k$.
  Then we see that for a holomorphic map $\varphi: \Delta \to \Omega_{2d}$ with $\varphi(0) = P$, 
  if $\varphi(t) = (\varphi_1(t), \ldots, \varphi(t)_{19})$ under the coordinate, the lifting $\varphi^{(1)}(t): \Delta \to \Omega_{2d}^{(1)}$ is
  \begin{equation}
    \varphi^{(1)} (t)  = (\varphi_1(t) , \varphi_2(t) / \varphi_1(t), \ldots, \varphi_{j,1}(t), \varphi_{j,2}(t)/\varphi_{j,1}(t), \ldots, ) 
  \end{equation}
In particular, $\varphi^{(1)}(0) \in b^{(1), -1}(P) \cong \mathbb{P}(\h_1) \times \cdots \times \mathbb{P}(\h_n)$ is the n-tuple of the roots of the bubbling trees $\mathcal{T}_1, \ldots, \mathcal{T}_n$ of $\varphi$ at each singularities corresponding to $S_j$.
Similarly, by replacing the coordinate by $\{z_{j,k,l}\}$ so that $\{z_{j,k,l} \}_{l=1}^{n_{j,k}} \subset \{z_{j,k}\}_{k=1}^{n_j}$ defines a subvariety 
\begin{equation}
    S_{j,k} = [z_{j,k,1} = \cdots = z_{j,k,n_{j,k}} =0] \supset S_j
\end{equation}
which is a minimal analytic set of ADE type containing $S_j$.
Then we see that $\varphi^{(2)}(t): \Delta \to \Omega_{2d}^{(2)}$, the lifting of $\varphi^{(1)}$, is in a form of 
\begin{equation}
    \varphi^{(2)}(t) = (\ldots, \varphi_{j,k,l}(t)/ \varphi_{j,k,1}(t), \ldots ),
\end{equation}
and hence $\varphi^{(2)}(0)$ is the tuple of the second descendants of the trees.
Inductively, it follows that $\varphi^{(l)}(0)$ takes its value the $l$-th descendants of the trees.
By the definition of the equivalence $\varphi \sim \psi$ and the above argument, it follows that $\varphi \sim \psi$ if and only if $\hat{\varphi}(0) = \hat{\psi}(0)$.
\end{proof}
\textbf{Question}: Can we construct a \textit{universal family} on $\hat{\mathcal{F}}_{2d}$ in a suitable sense ?

See also \cite{dBS} Section 2 (subsection 2.4. in particular) which gives a picture to regard the Deligne-Mumford compactification of the moduli space of $\mathbb{P}^1$ with $n$-points as the \textit{multiscale K-moduli space} of K\"ahler-Einstein metrics on $\mathbb{P}^1$ with cone angles at $n$-points.

\appendix 
\section{Notes on ALE hyperk\"ahler gravitational instantons}
In this appendix, some propositions on ALE hyperk\"ahler gravitational instantons are presented  sketches of proofs are given for readers convenience.
The author believes these facts are well-known for experts but he can not find proofs in the literature.
In the rest of this appendix, $\h_\R^{\oplus 3}$ is identified as $\h_\R \oplus \h_\C$ in the natural way and for a hyperk\"ahler manifold $(M,g,I,J,K)$, it is regarded as a complex manifold by the complex structure $I$.
\begin{prop}\label{prop:appendix 1}
   Let $X$ be the underlying differentiable manifold of the minimal resolution of $\C^2/\Gamma$ where $\Gamma \subset \mathrm{SL}(2;\C)$.
   For a triple $\bm{\kappa} = (\kappa_1, \kappa_2, \kappa_3) \in H^2(X,\Z)^{\oplus 3}$, let $R(\bm{\kappa}) \subset H(X;\Z)$ be the subset of roots perpendicular to $\bm{\kappa}$: 
   \begin{equation}
       R(\bm{\kappa}) = \{\theta \in H^2(X;\Z) \mid \theta^2 =-2, (\theta, \kappa_j) =0\ (j=1,2,3)\}.
   \end{equation}
   Then there exists a unique hyperk\"ahler orbifold $X_{\bm{\kappa}} =(X_{\bm{\kappa}}, g_{\bm{\kappa}}, I_{\bm{\kappa}}, J_{\bm{\kappa}}, K_{\bm{\kappa}})$ such that
   \begin{itemize}
       \item $X_{\bm{\kappa}}$ is a deformation of $\C^2/\Gamma$. In particular, the minimal resolution $\rho: \widetilde{X}_{\bm{\kappa}} \to X_{\bm{\kappa}}$ is diffeomorphic to $X$ and
       \item the cohomology classes of $\rho^*\bm{\omega} = (\rho^*\omega_{\bm{\kappa},I}, \rho^*\omega_{\bm{\kappa},J}, \rho^*\omega_{\bm{\kappa},K})$ is equal to $\bm{\kappa}$ where $\bm{\omega}_{\bm{\kappa}}$ is the triple of the K\"ahler forms. In particular,
       \begin{equation}
           R(\bm{\kappa}) = \{\theta = \sum n_j [E_j] \mid \theta^2 =-2, n_j \in \Z \},
       \end{equation}
       where $\{E_j\}$ is the irreducible components of the exceptional divisors.
   \end{itemize}
\end{prop}
\begin{proof}
    In the proof of Lemma 3.3.\cite{Kro1}, it is proved that on discriminant loci, orbifold ALE hyperk\"ahler gravitational instantons are parameterized (though only homeomorphism is asserted in the statement).
    And the rest is exactly a summary of the results of the rest of the paper \cite{Kro1}. 
    
\end{proof}

\begin{prop}\label{prop:appendix 5}
    Let $Y_1$ and $Y_2$ be affine ALE gravitational instantons constructed in Proposition \ref{prop:appendix 1}.
    Let $Y_1$ corresponds to $\kappa_1$ and $Y_2$ corresponds to $\kappa_2$. 
    Then they are isomorphic if and only if there are an automorphism $\alpha \in \mathrm{O}(H^2(\widetilde{\C^2/\Gamma};\Z))$ and a constant $c \in \C \setminus\{0\}$ such that 
    \begin{equation}
        \alpha(\kappa_1) = c \kappa_2 .
    \end{equation}
\end{prop}
\begin{proof}
    The proof of the Torelli theorem for ALE hyperk\"ahler gravitational instantons given in \cite{Kro2} is completely valid for orbifolds. 
    Hence two orbifolds $Y_1$ and $Y_2$ have the same (up to gauge) hyperk\"ahler triples if and only if they corresponds points in the same $\mathrm{O}(H^2(\widetilde{\C^2/\Gamma} ;\Z))$ orbit.
    On the other hand, a natural $\C^*$ action does not change holomorphic structure, then the assertion follows.
\end{proof}

\begin{prop}\label{prop:appendix2}
    Let $X$ be a (non-compact) complete hyperk\"ahler orbifold with only singularities isomorphic to $0 \in \C^2/\Gamma$ for $\Gamma \subset \mathrm{SL}(2;\C)$.
    Assume that $X$ is ALE of order $4$.
    Then $X$ is isomorphic to an orbifold in Proposition \ref{prop:appendix 1} as hyperk\"ahler orbifolds.
\end{prop}
\begin{proof}
    By \cite{Bando} Theorem 4, the minimal resolution of $B$ admits a structure of ALE hyperk\"ahler gravitational instanton. 
    In particular $B$ is diffeomorphic to the minimal resolution of $\C^2/\Gamma$ for some $\Gamma \subset \mathrm{SL}(2;\C)$.
    Then there is an orbifold ALE hyperk\"ahler gravitational instanton $X_{\bm{\kappa}}$ where $\bm{\kappa}$ is the triple of the cohomology classes of the triple of pull-back of K\"ahler forms on $B$.
    $X_{\bm{\kappa}}$ must be isomorphic to $B$ by Proposition \ref{prop:appendix 5} .
\end{proof}
\begin{cor}
    Let $B$ be a bubbling limit of a non-collapsing sequence of Ricci-flat K3 surfaces.
    Then $B$ is one of the orbifold appearing above.
\end{cor}
\begin{proof}
    By \cite{BKN}, $B$ satisfies the assumption of the above Proposition \ref{prop:appendix2}.
\end{proof}

\bibliographystyle{plain}
\bibliography{Bibliography}

\end{document}